\documentclass{m2an}
\usepackage{enumerate}
\usepackage{subfigure}
\usepackage{graphicx}
\usepackage{float}
\usepackage{color}
\newtheorem{theorem}{Theorem}[section]
\newtheorem{lemma}[theorem]{Lemma}

\theoremstyle{definition}

\newtheorem{example}[theorem]{Example}

\theoremstyle{remark}
\newtheorem{remark}[theorem]{Remark}

\numberwithin{equation}{section}
\begin{document}
\title{Numerical algorithm for the model describing anomalous diffusion in expanding media}
\author{Daxin Nie}\address{School of Mathematics and Statistics, Gansu Key Laboratory of Applied Mathematics and Complex Systems, Lanzhou University, Lanzhou 730000, P.R. China}

\author{Jing Sun}\address{School of Mathematics and Statistics, Gansu Key Laboratory of Applied Mathematics and Complex Systems, Lanzhou University, Lanzhou 730000, P.R. China}

\author{Weihua Deng}\address{School of Mathematics and Statistics, Gansu Key Laboratory of Applied Mathematics and Complex Systems, Lanzhou University, Lanzhou 730000, P.R. China}\email{dengwh@lzu.edu.cn}

\date{}
\begin{abstract}
	We provide a numerical algorithm for the model characterizing anomalous diffusion in expanding media, which is derived in [F. Le Vot, E. Abad, and S. B. Yuste,  Phys. Rev. E {\bf96} (2017) 032117]. The Sobolev regularity for the equation is first established. Then we use the finite element method to discretize the Laplace operator and present error estimate of the spatial semi-discrete scheme based on the regularity of the solution; the backward Euler convolution quadrature is developed to approximate Riemann-Liouville fractional derivative and error estimates for the fully discrete scheme are established by using the continuity of solution. Finally, the numerical experiments verify the effectiveness of the algorithm.
 \end{abstract}
\subjclass{65M60, 42A85, 35R11}
\keywords{fractional diffusion equation, variable coefficient, finite element method, convolution quadrature, error analysis}
\maketitle
\section{Introduction}
Currently, it is widely recognized that anomalous diffusions are ubiquitous in the natural world, and some important models are built, including the continuous time random walk (CTRW) model, e.g., \cite{Barkai2000,Barkai2001,Xu2018}, and the Langevin picture, e.g., \cite{Chen20191}. Most of the CTRW models mimic anomalous diffusion processes in static media, while expanding media are typical in biology and cosmology.  Recently, \cite{LeVot2017} builds the CTRW model for anomalous diffusion in expanding media, the Langevin picture of which is given in \cite{Chen20192}, and derives the corresponding Fokker-Planck equation
%
%
\begin{equation}\label{eqretosol}
\left\{
\begin{aligned}
&\frac{\partial W(x,t)}{\partial t}=\frac{1}{a^2(t)}\Delta\left[\,_0D^{1-\alpha}_tW(x,t)\right]+f(x,t), \qquad(x,t)\in \Omega\times(0,T],\\
&W(x,0)=W_0(x),\qquad\qquad\qquad\qquad\qquad\qquad\qquad\quad  x\in \Omega,\\
&W(x,t)=0,\qquad\qquad\quad\qquad\qquad\qquad\qquad\qquad\quad\quad (x,t)\in \partial \Omega\times(0,T],
\end{aligned}
\right.
\end{equation}
where $\Delta$ stands for Laplace operator; $f(x,t)$ is the source term; $\Omega\subset \mathbb{R}$ is a bounded domain; $T$ is a fixed final time; $_0D^{1-\alpha}_t$ denotes the Riemann-Liouville fractional derivative, defined as \cite{Podlubny1999}
\begin{equation*}
\!_0D^{1-\alpha}_t W(t)=\frac{\partial}{\partial t}\,_0I_t^{\alpha}W(t)=\frac{1}{\Gamma(\alpha)}\frac{\partial}{\partial t}\int_0^t\frac{W(\xi)}{(t-\xi)^{1-\alpha}}d\xi, \qquad 0<\alpha<1,
\end{equation*}
and $\,_0I^\alpha_t$ denotes the Riemann-Liouville fractional integral; $a^2(t)$ is the variable diffusion coefficient satisfying
\begin{equation}\label{eqassuma}
	\left|\frac{1}{a^2(t)}-\frac{1}{a^2(s)}\right|\leq C|t-s|,\quad t,s\in[0,T]
\end{equation}
and
\begin{equation}\label{eqassumb}
c<a^2(t)<C,\quad t\in[0,T]
\end{equation}
with $c$ and $C$ being two positive constants.

 So far, numerical methods for fractional differential equations have gained widespread concerns\cite{Bazhlekova2015,Chen2009,Deng2009,Deng2013,Jin2014,Jin2015,Jin2016,Jin2017,Li2009,Lin2007,Zeng2018}, and \cite{Jin2019,Mustapha2018} also provide a complete numerical analysis for fractional differential equations with variable coefficient. Compared with them, the non-commutativity of the Riemann-Liouville fractional derivative and the variable coefficient, i.e., $\frac{1}{a^2(t)}~_0D^{1-\alpha}_t\neq ~_0D^{1-\alpha}_t\frac{1}{a^2(t)}$, brings new challenges in the priori estimate and numerical analysis. To obtain the priori estimate of the solution $W(x,t)$ of Eq. \eqref{eqretosol}, the regularity of $\,_0D^{1-\alpha}_tW(x,t)$ is needed. As for the spatial discretization, we use finite element method to discretize Laplace operator $\Delta$ and get the optimal-order convergence rates. And then we use backward Euler convolution quadrature \cite{Lubich1988,Lubich19882} to discretize Riemann-Liouville fractional derivative and derive error estimates for fully discrete scheme by using H\"{o}lder continuity.

 The rest of the paper is organized as follows. We first provide some preliminaries and then give some priori estimates for the solution of Eq. \eqref{eqretosol} in Sec. 2. In Sec. 3, we use the finite element method to discretize the Laplace operator and get the error estimate of the spatial semi-discrete scheme. Section 4 approximates the Riemann-Liouville fractional derivative by backward Euler convolution quadrature and gives the error estimates of the fully discrete scheme for the homogeneous and inhomogeneous problems. In the last section, we verify the effectiveness of the algorithm by numerical experiments.
\noindent{\section{Preliminary}}
We first give some preliminaries. For $\kappa>0$ and $\pi/2<\theta<\pi$, we define sectors $\Sigma_{\theta}$ and $\Sigma_{\theta,\kappa}$ in the complex plane $\mathbb{C}$ as
\begin{equation*}
	\begin{aligned}
		&\Sigma_{\theta}=\{z\in\mathbb{C}\setminus \{0\},|\arg z|\leq \theta\}, \quad&\Sigma_{\theta,\kappa}=\{z\in\mathbb{C}:|z|\geq\kappa,|\arg z|\leq \theta\},\\
	\end{aligned}
\end{equation*}
and the contour $\Gamma_{\theta,\kappa}$ is defined by
\begin{equation*}
	\Gamma_{\theta,\kappa}=\{z\in\mathbb{C}: |z|=\kappa,|\arg z|\leq \theta\}\cup\{z\in\mathbb{C}: z=r e^{\pm \mathbf{i}\theta}: r\geq \kappa\},
\end{equation*}
oriented with an increasing imaginary part, where $\mathbf{i}$ denotes the imaginary unit and $\mathbf{i}^2=-1$. We use  $\|\cdot\|$ to denote the operator norm from $L^2(\Omega)$ to $L^2(\Omega)$ and $\epsilon$ any small number larger than $0$.

Then we introduce $G(x,t)=\!_0D^{1-\alpha}_tW(x,t)$, $A=-\Delta$, and $A(t)=-\frac{1}{a^2(t)}\Delta$. For any $ r\geq  0 $, denote the space $ \dot{H}^r(\Omega)=\{v\in L^2(\Omega): A^{\frac{r}{2}}v\in L^2(\Omega) \}$ with the norm \cite{Bazhlekova2015}
	\begin{equation*}
		\|v\|^2_{\dot{H}^r(\Omega)}=\sum_{j=1}^{\infty}\lambda_j^r(v,\varphi_j)^2,
	\end{equation*}
	where $ {(\lambda_j,\varphi_j)} $ are the eigenvalues ordered non-decreasingly and the corresponding eigenfunctions normalized in the $ L^2(\Omega) $ norm of operator $A$ subject to the homogeneous Dirichlet boundary conditions on $\Omega$. Thus $ \dot{H}^0(\Omega)=L^2(\Omega) $, $\dot{H}^1(\Omega)=H^1_0(\Omega)$, and $\dot{H}^2(\Omega)=H^2(\Omega)\bigcap H^1_0(\Omega)$.
 For simplicity, we denote $G(t)$, $W(t)$, $W_0$, and $f(t)$ as $G(x,t)$, $W(x,t)$, $W_0(x)$, and $f(x,t)$ in the following. Throughout this paper, $C$ denotes a generic positive constant, whose value may differ at each occurrence.

 According to \eqref{eqassuma}, there exists
\begin{equation}\label{eqassumae}
	\|(A(t)-A(s))u\|_{L^2(\Omega)}\leq C|t-s|\|u\|_{\dot{H}^2(\Omega)}.
\end{equation}
Thus by simple calculations, for any fixed $t_0\in (0,T]$, the solution of Eq. \eqref{eqretosol} can be represented as
\begin{equation}\label{eqrepsW}
	\begin{aligned}
		W(t)=F(t,t_0)W_0+\int_0^tF(t-s,t_0)f(s)ds+\int_0^tF(t-s,t_0)\left (A(t_0)-A(s)\right )G(s)ds,
	\end{aligned}
\end{equation}
where
\begin{equation}\label{eqdefF}
	F(t,t_0):=\frac{1}{2\pi \mathbf{i}}\int_{\Gamma_{\theta,\kappa}}e^{zt}z^{\alpha-1}(z^\alpha+A(t_0))^{-1}dz.
\end{equation}
By means of the Laplace transform and the definition of $G(t)$, we get
\begin{equation}\label{eqrepsG}
	G(t)=E(t,t_0)W_0+\int_0^tE(t-s,t_0)f(s)ds+\int_0^tE(t-s,t_0)(A(t_0)-A(s))G(s)ds,
\end{equation}
where
\begin{equation}\label{eqdefE}
	E(t,t_0):=\frac{1}{2\pi \mathbf{i}}\int_{\Gamma_{\theta,\kappa}}e^{zt}(z^\alpha+A(t_0))^{-1}dz.
\end{equation}
As for the operators   $F(t,t_0)$ and $E(t,t_0)$, there exist the following estimates.
\begin{lemma}[\cite{Jin2019}]\label{lemestEF}
The operators $F(t,t_0)$ and  $E(t,t_0)$  defined in \eqref{eqdefF} and  \eqref{eqdefE}  satisfy
\begin{equation*}
	\begin{aligned}
		&\|E(t,t_0)\|\leq Ct^{\alpha-1},\quad \|F(t,t_0)\|\leq C,\quad\|A^{1-\beta}E(t,t_0)\|\leq Ct^{\alpha\beta-1},\\
		&\|A^{\beta}F(t,t_0)\|\leq Ct^{-\alpha\beta},\quad	\|A^{-\beta}F'(t,t_0)\|\leq Ct^{\alpha\beta-1},
	\end{aligned}
\end{equation*}
where $F'(t,t_0)$ denotes the first derivative about $t$ and $\beta\in[0,1]$.
\end{lemma}
\begin{remark}
The estimates in Lemma \ref{lemestEF} are got by mainly using $\|(z+A)^{-1}\|\leq C|z|^{-1}$ for $z\in \Sigma_\theta$. And the last estimate can be obtained by using the fact $z^\alpha(z^{\alpha}+A)^{-1}=\mathbf{I}-A(z^{\alpha}+A)^{-1}$, where $\mathbf{I}$ denotes the identity operator.
\end{remark}
To get the priori estimate of $W(t)$, we first provide some estimates of $G(t)$.
\begin{theorem}\label{thmregofG}
	If $W_0\in \dot{H}^{\epsilon}(\Omega)$, $f(0)\in L^2(\Omega)$ and $\int_{0}^{t}\|f'(s)\|_{L^2(\Omega)}ds<C$ with $t\in [0,T]$, then $G(t)$ satisfies
	\begin{equation*}
			\|G(t)\|_{L^2(\Omega)}\leq C(T)t^{\alpha-1}\|W_0\|_{L^2(\Omega)}+C(T)\|f(0)\|_{L^2(\Omega)}+C(T)\int_{0}^{t}\|f'(s)\|_{L^2(\Omega)}ds
	\end{equation*}
	and
	\begin{equation*}
		\|G(t)\|_{\dot{H}^2(\Omega)}\leq C(T)t^{\epsilon\alpha/2-1}\|W_0\|_{\dot{H}^{\epsilon}(\Omega)}+C(T)\|f(0)\|_{L^2(\Omega)}+C(T)\int_{0}^{t}\|f'(s)\|_{L^2(\Omega)}ds.
	\end{equation*}
\end{theorem}
\begin{proof}
	Applying $A(t_0)$ on both sides of \eqref{eqrepsG} and taking $L^2$ norm on both sides yield
	\begin{equation*}
		\begin{aligned}
			\|A(t_0)G(t_0)\|_{L^2(\Omega)}\leq&\|A(t_0)E(t_0,t_0)W_0\|_{L^2(\Omega)}+\left\|\int_0^{t_0}A(t_0)E(t_0-s,t_0)f(s)ds\right\|_{L^2(\Omega)}\\
			&+\left\|\int_0^{t_0}A(t_0)E(t_0-s,t_0)(A(t_0)-A(s))G(s)ds\right\|_{L^2(\Omega)}.
		\end{aligned}
	\end{equation*}
According to Lemma \ref{lemestEF}, \eqref{eqassumae}, and convolution properties, there exists 
	\begin{equation*}
		\begin{aligned}
			\|A(t_0)G(t_0)\|_{L^2(\Omega)}\leq& Ct_0^{\epsilon\alpha/2-1}\|W_0\|_{\dot{H}^{\epsilon}(\Omega)}+C\|f(0)\|_{L^2(\Omega)}+C\int_{0}^{t_0}\|f'(s)\|_{L^2(\Omega)}ds+C\int_0^{t_0} \|G(s)\|_{\dot{H}^2(\Omega)}ds.
		\end{aligned}
	\end{equation*}
	Taking $t_0=t$ and using Gr\"{o}nwall's inequality \cite{Jin2019,Larsson1992} lead to
	\begin{equation*}
		\|G(t)\|_{\dot{H}^2(\Omega)}\leq C(T)t^{\epsilon\alpha/2-1}\|W_0\|_{\dot{H}^{\epsilon}(\Omega)}+C(T)\|f(0)\|_{L^2(\Omega)}+C(T)\int_{0}^{t}\|f'(s)\|_{L^2(\Omega)}ds.
	\end{equation*}
	Similarly we have
	\begin{equation*}
		\|G(t)\|_{L^2(\Omega)}\leq C(T)t^{\alpha-1}\|W_0\|_{L^2(\Omega)}+C(T)\|f(0)\|_{L^2(\Omega)}+C(T)\int_{0}^{t}\|f'(s)\|_{L^2(\Omega)}ds.
	\end{equation*}
\end{proof}
Next we give the regularity estimate of $W(t)$.

\begin{theorem}\label{thmregofW}
	If $W_0\in \dot{H}^{\epsilon}(\Omega)$, $f(0)\in L^2(\Omega)$ and $\int_{0}^{t}\|f'(s)\|_{L^2(\Omega)}ds<C$ with $t\in [0,T]$, then the solution $W(t)$ of \eqref{eqretosol} satisfies
	\begin{equation*}
		\begin{aligned}
			\|W(t)\|_{\dot{H}^2(\Omega)}
			\leq
			& C(T)t^{-\alpha}\|W_0\|_{\dot{H}^\epsilon(\Omega)}+C(T)\|f(0)\|_{L^2(\Omega)}+C(T)\int_{0}^{t}\|f'(s)\|_{L^2(\Omega)}ds.
		\end{aligned}
	\end{equation*}
\end{theorem}
\begin{proof}
	Applying $A(t_0)$ on both sides of \eqref{eqrepsW} and taking $L^2$ norm lead to
	\begin{equation*}
		\begin{aligned}
			\|A(t_0)W(t_0)\|_{L^2(\Omega)}\leq&\left \|A(t_0)F(t_0,t_0)W_0\right \|_{L^2(\Omega)}+\left\|\int_0^{t_0}A(t_0)F(t_0-s,t_0)f(s)ds\right \|_{L^2(\Omega)}\\
			&+\left \|\int_0^{t_0}A(t_0)F(t_0-s,t_0)(A(t_0)-A(s))G(s)ds\right \|_{L^2(\Omega)}.\\
		\end{aligned}
	\end{equation*}
According to (\ref{eqassumae}), Lemma \ref{lemestEF}, and the fact $T/(t_0-s)>1$, there exists
	\begin{equation*}
		\begin{aligned}
		&	\|A(t_0)W(t_0)\|_{L^2(\Omega)} \\ & \leq Ct_0^{-\alpha}\|W_0\|_{L^2(\Omega)}+C\|f(0)\|_{L^2(\Omega)}+C\int_0^{t_0}\|f'(s)\|_{L^2(\Omega)}ds+C\int_0^{t_0}(t_0-s)^{1-\alpha}\|G(s)\|_{\dot{H}^2(\Omega)}ds.
		\end{aligned}
	\end{equation*}
Further combining Theorem \ref{thmregofG} results in
	\begin{equation*}
		\begin{aligned}
			\|W(t_0)\|_{\dot{H}^2(\Omega)}
			\leq& C(T)t_0^{-\alpha}\|W_0\|_{\dot{H}^\epsilon(\Omega)}+C(T)\|f(0)\|_{L^2(\Omega)}+C(T)\int_{0}^{t_0}\|f'(s)\|_{L^2(\Omega)}ds,
		\end{aligned}
	\end{equation*}
which leads to the desired result after taking $t_0=t$,.
\end{proof}

\noindent{\section{Spacial discretization and error analysis}}
In this section, we discretize Laplace operator by the finite element method and provide the error estimates for the space semi-discrete scheme of Eq. \eqref{eqretosol}. Let $\mathcal{T}_h$ be a shape regular quasi-uniform partitions of the domain $\Omega$, where $h$ is the maximum diameter. Denote $ X_h $ as piecewise linear finite element space
\begin{equation*}
X_{h}=\{v_h\in C(\bar{\Omega}): v_h|_\mathbf{T}\in \mathcal{P}^1,\  \forall \mathbf{T}\in\mathcal{T}_h,\ v_h|_{\partial \Omega}=0\},
\end{equation*}
where $\mathcal{P}^1$ denotes the set of piecewise polynomials of degree $1$ over $\mathcal{T}_h$. Then we define the $ L^2 $-orthogonal projection $ P_h:\ L^2(\Omega)\rightarrow X_h $ and the Ritz projection $ R_h:\ H^1_0(\Omega)\rightarrow X_h $ \cite{Bazhlekova2015}, respectively, by 
\begin{equation*}
\begin{aligned}
&(P_hu,v_h)=(u,v_h) \ ~~~\forall v_h\in X_h,\\
&(\nabla R_h u,\nabla v_h)=(\nabla u, \nabla v_h) \ ~~~\forall v_h\in X_h.
\end{aligned}
\end{equation*}
\begin{lemma}[\cite{Bazhlekova2015}]\label{lemprojection}
	The projections $ P_h $ and $ R_h $ satisfy
	\begin{equation*}
	\begin{aligned}
	&\|P_hu-u\|_{L^2(\Omega)}+h\|\nabla(P_hu-u)\|_{L^2(\Omega)}\leq Ch^q\|u\|_{\dot{H}^{q}(\Omega)}\ {\rm for}\ u\in \dot{H}^q(\Omega),\ q=1,2,\\
	&\|R_hu-u\|_{L^2(\Omega)}+h\|\nabla(R_hu-u)\|_{L^2(\Omega)}\leq Ch^q\|u\|_{\dot{H}^{q}(\Omega)}\ {\rm for}\ u\in \dot{H}^q(\Omega),\ q=1,2.
	\end{aligned}
	\end{equation*}
\end{lemma}
Denote $(\cdot,\cdot)$ as the $L_2$ inner product and $A_h$ defined from $(A_hu,v)=(\nabla u, \nabla v)$. The semi-discrete Galerkin scheme for \eqref{eqretosol} reads: For every $t\in (0,T]$ find $ W_{h}\in X_h$ such that 
\begin{equation}\label{eqsemischeme}
	\left \{
	\begin{aligned}
		&\left(\frac{\partial W_h}{\partial t},v\right)+(\,_0D^{1-\alpha}_tA_h(t_0)W_h,v)=(f,v)+((A_h(t_0)-A_h(t))\,_0D^{1-\alpha}_tW_h,v)\quad {\rm~for~all~} v\in X_h,
		\\
		&W_h(0)=W_{0,h}, 
		\\
	\end{aligned}
	\right.
\end{equation}
where 
\begin{equation*}
	W_{0,h}=\left \{
	\begin{aligned}
		&P_hW_0,\qquad W_0\in L^2(\Omega),\\
		&R_hW_0,\qquad W_0\in \dot{H}^2(\Omega),
	\end{aligned}\right .
\end{equation*}
 and
\begin{equation*}
	 (A_h(t)u,v)=\frac{1}{a^2(t)}(\nabla u, \nabla v).
\end{equation*}
For convenience, we rewrite the spatial semi-discrete scheme as
\begin{equation*}
	\frac{\partial W_h}{\partial t}+\,_0D^{1-\alpha}_tA_h(t_0)W_h=f_h+(A_h(t_0)-A_h(t))\,_0D^{1-\alpha}_tW_h,
\end{equation*}
where $f_h=P_hf$. By means of Laplace transform, the solution of \eqref{eqsemischeme} can be rewritten as
\begin{equation}\label{eqrepsWh}
	\begin{aligned}
		W_h(t)=&F_h(t,t_0)W_{0,h}+\int_0^tF_h(t-s,t_0)f_h(s)ds+\int_0^tF_h(t-s,t_0)(A_h(t_0)-A_h(s))\,_0D^{1-\alpha}_sW_h(s)ds,
	\end{aligned}
\end{equation}
where
\begin{equation}\label{equdefFh}
	F_h(t,t_0):=\frac{1}{2\pi \mathbf{i}}\int_{\Gamma_{\theta,\kappa}}e^{zt}z^{\alpha-1}(z^\alpha+A_h(t_0))^{-1}dz.
\end{equation}
Introducing $G_h(t)=\,_0D^{1-\alpha}_tW_h(t)$, thus $G_h(t)$ can be represented by
\begin{equation}\label{eqrepsGh}
	\begin{aligned}
		G_h(t)=&E_h(t,t_0)W_{0,h}+\int_0^tE_h(t-s,t_0)f_h(s)ds+\int_0^tE_h(t-s,t_0)(A_h(t_0)-A_h(s))G_h(s)ds,
	\end{aligned}
\end{equation}
where
\begin{equation}\label{equdefEh}
	E_h(t,t_0):=\frac{1}{2\pi \mathbf{i}}\int_{\Gamma_{\theta,\kappa}}e^{zt}(z^\alpha+A_h(t_0))^{-1}dz.
\end{equation}

Similar to Lemma \ref{lemestEF}, the following estimates about $E_h$ and $F_h$ hold.
\begin{lemma}[\label{lemestEFh}\cite{Jin2019}]
	The operators $F_h(t,t_0)$ and $E_h(t,t_0)$ defined in \eqref{equdefFh} and \eqref{equdefEh} satisfy
	\begin{equation*}
	\begin{aligned}
	&\|E_h(t,t_0)\|\leq Ct^{\alpha-1},\quad \|F_h(t,t_0)\|\leq C,\quad\|A_h^{1-\beta}E_h(t,t_0)\|\leq Ct^{\alpha\beta-1},\\
	&\|A_h^{\beta}F_h(t,t_0)\|\leq Ct^{-\alpha\beta},\quad	\|A_h^{-\beta}F_h'(t,t_0)\|\leq Ct^{\alpha\beta-1},
	\end{aligned}
	\end{equation*}
where $\beta\in [0,1]$.
\end{lemma}

Next, we provide the following Lemma which helps us for the error estimate.

\begin{lemma}[\label{lemerror1}\cite{Bazhlekova2015}]
	Let $\phi\in L^2(\Omega)$, $z\in \Sigma_{\theta}$, $\omega=(z^\alpha \mathbf{I}+A)^{-1}\phi$, and $\omega_h=(z^\alpha \mathbf{I}+A_h)^{-1}P_h\phi$, where $\mathbf{I}$ denotes the identity operator. Then there holds
	\begin{equation*}
		\|\omega_h-\omega\|_{L^2(\Omega)}+h\|\nabla(\omega_h-\omega)\|_{L^2(\Omega)}\leq Ch^2\|\phi\|_{L^2(\Omega)}.
	\end{equation*}
\end{lemma}
To get the error estimate for the space semi-discrete scheme. Denote $e_h(t)=P_hW(t)-W_h(t)$.  From \eqref{eqrepsW} and \eqref{eqrepsWh}, there exists
\begin{equation}\label{eqehsep}
	\begin{aligned}
		e_h(t)=&(P_hF(t,t_0)W_0-F_h(t,t_0)W_{0,h})+\int_0^t(P_hF(t-s,t_0)-F_h(t-s,t_0)P_h)f(s)ds\\
		&+\int_0^t(P_hF(t-s,t_0)-F_h(t-s,t_0)P_h)(A(t_0)-A(s))\,_0D^{1-\alpha}_sW(s)ds\\
		&+\int_0^tF_h(t-s,t_0)((P_hA(t_0)-P_hA(s))\,_0D^{1-\alpha}_sW(s)-(A_h(t_0)-A_h(s))\,_0D^{1-\alpha}_sW_h(s))ds\\
		=&\uppercase\expandafter{\romannumeral1}(t)+\uppercase\expandafter{\romannumeral2}(t)+\uppercase\expandafter{\romannumeral3}(t)+\uppercase\expandafter{\romannumeral4}(t).
	\end{aligned}
\end{equation}

Then we need to provide the bounds of $\uppercase\expandafter{\romannumeral1}(t)$, $\uppercase\expandafter{\romannumeral2}(t)$, $\uppercase\expandafter{\romannumeral3}(t)$, and  $\uppercase\expandafter{\romannumeral4}(t)$ in \eqref{eqehsep}.

\begin{lemma}\label{lemspaestI}
	If $W_0\in L^2(\Omega)$, there exists the estimate
	\begin{equation*}
		\begin{aligned}
			\|\uppercase\expandafter{\romannumeral1}(t)\|_{L^2(\Omega)}\leq& Ct^{-\alpha}h^2\|W_0\|_{L^2(\Omega)}.\\
		\end{aligned}
	\end{equation*}
\end{lemma}
\begin{proof}
	According to Lemmas \ref{lemprojection} and \ref{lemerror1}, 
	\begin{equation*}
		\begin{aligned}
			\|\uppercase\expandafter{\romannumeral1}(t)\|_{L^2(\Omega)}\leq& \|(P_hF(t,t_0)-F_h(t,t_0)P_h)W_0\|_{L^2(\Omega)}\\
			\leq& \|(P_hF(t,t_0)-F(t,t_0))W_0\|_{L^2(\Omega)}+\|(F(t,t_0)-F_h(t,t_0)P_h)W_0\|_{L^2(\Omega)}\\
			\leq& C(T)t^{-\alpha}h^2\|W_0\|_{L^2(\Omega)},
		\end{aligned}
	\end{equation*}
	which leads to the desired result.
\end{proof}
Similarly, we have the following estimate of $\uppercase\expandafter{\romannumeral2}(t)$.
\begin{lemma}\label{lemspaestII}
	If $f(0)\in L^2(\Omega)$ and $\int_0^t\|f'(s)\|_{L^2(\Omega)}ds<\infty$, then $\uppercase\expandafter{\romannumeral2}(t)$ can be bounded by
	\begin{equation*}
		\begin{aligned}
			\|\uppercase\expandafter{\romannumeral2}(t)\|_{L^2(\Omega)}\leq& C(T)h^2\left (\|f(0)\|_{L^2(\Omega)}+\int_{0}^{t}\|f'(s)\|_{L^2(\Omega)}ds\right ).\\
		\end{aligned}
	\end{equation*}
\end{lemma}
As for $\uppercase\expandafter{\romannumeral3}(t)$, there exists the estimate
\begin{lemma}\label{lemspaestIII}
	If $W_0\in \dot{H}^\epsilon(\Omega)$, $f(0)\in L^2(\Omega)$, and $\int_0^t\|f'(s)\|_{L^2(\Omega)}ds<\infty$, then 
	\begin{equation*}
		\|\uppercase\expandafter{\romannumeral3}(t)\|_{L^2(\Omega)}\leq C(T)h^2\left (\|W_0\|_{\dot{H}^{\epsilon}(\Omega)}+\|f(0)\|_{L^2(\Omega)}+\int_{0}^{t}\|f'(s)\|_{L^2(\Omega)}ds\right ).
	\end{equation*}
\end{lemma}
\begin{proof}
	According to \eqref{eqassumae}, Lemmas \ref{lemprojection}, \ref{lemerror1}, and Theorem \ref{thmregofG}, we have
	\begin{equation*}
		\begin{aligned}
			\|\uppercase\expandafter{\romannumeral3}(t_0)\|_{L^2(\Omega)}\leq&\int_0^{t_0}\|P_hF(t_0-s,t_0)-F(t_0-s,t_0)\|\|(A(t_0)-A(s))\,_0D^{1-\alpha}_sW(s)\|_{L^2(\Omega)}ds\\
			&+\int_0^{t_0}\|F(t_0-s,t_0)-F_h(t_0-s,t_0)P_h\|\|(A(t_0)-A(s))\,_0D^{1-\alpha}_sW(s)\|_{L^2(\Omega)}ds\\
			\leq&Ch^2\int_0^{t_0}(t_0-s)^{1-\alpha}\|\,_0D^{1-\alpha}_sW(s)\|_{\dot{H}^2(\Omega)}ds\\
			\leq&Ch^2\left (\|W_0\|_{\dot{H}^{\epsilon}(\Omega)}+\|f(0)\|_{L^2(\Omega)}+\int_{0}^{t_0}\|f'(s)\|_{L^2(\Omega)}ds\right ).
		\end{aligned}
	\end{equation*}
	Taking $t_0=t$ leads to the desired result.
\end{proof}
 To estimate $\|\uppercase\expandafter{\romannumeral4}(t)\|_{L^2(\Omega)}$,  introducing $\upsilon_h(t)=\,_0D^{1-\alpha}_te_h$ results in
\begin{equation*}
	\begin{aligned}
		\upsilon_h(t)=&(P_hE(t,t_0)W_0-E_h(t,t_0)W_{0,h})+\int_0^t(P_hE(t-s,t_0)-E_h(t-s,t_0)P_h)f(s)ds\\
		&+\int_0^t(P_hE(t-s,t_0)-E_h(t-s,t_0)P_h)(A(t_0)-A(s))G(s)ds\\
		&+\int_0^tE_h(t-s,t_0)((P_hA(t_0)-P_hA(s))G(s)-(A_h(t_0)-A_h(s))G_h(s))ds
		=\sum_{i=1}^{4}\upsilon_{i,h}(t).
	\end{aligned}
\end{equation*}

Next, we consider the estimate of $\|\upsilon_h(t)\|_{L^2(\Omega)}$, which helps to get the estimate of  $\|\uppercase\expandafter{\romannumeral4}(t)\|_{L^2(\Omega)}$.
\begin{lemma}\label{lemepslonest}
		If $W_0\in \dot{H}^{\epsilon}(\Omega)$, $f(0)\in L^2(\Omega)$ and $\int_{0}^{t}\|f'(s)\|_{L^2(\Omega)}ds<\infty$, then we have 
	\begin{equation*}
		\|\upsilon_h(t)\|_{L^2(\Omega)}\leq Ch^2t^{\epsilon\alpha/2-1}\|W_0\|_{\dot{H}^\epsilon(\Omega)}+Ch^2\|f(0)\|_{L^2(\Omega)}+Ch^2\int_{0}^{t}\|f'(s)\|_{L^2(\Omega)}ds.
	\end{equation*}
\end{lemma}
\begin{proof}
	According to Lemma \ref{lemprojection}, we have the estimates
	\begin{equation*}
		\left \|(P_hE(t,t_0)-E(t,t_0))W_0\right \|_{L^2(\Omega)}\leq\left \{\begin{aligned}
		&Ch^2t^{-1}\|W_0\|_{L^2(\Omega)},\qquad W_0\in L^2(\Omega),\\
		&Ch^2t^{\alpha-1}\|W_0\|_{\dot{H}^2(\Omega)},\qquad W_0\in \dot{H}^2(\Omega).
		\end{aligned}\right.
	\end{equation*}
	If $W_0\in L^2(\Omega)$, according to Lemma \ref{lemerror1}, the following estimate holds
	\begin{equation*}
		\begin{aligned}
		&\left \|(E(t,t_0)-E_h(t,t_0)P_h)W_0\right \|_{L^2(\Omega)}\\
		\leq& \left\| \int_{\Gamma_{\theta,\kappa}}e^{zt}((z^{\alpha}+A(t_0))^{-1}-(z^{\alpha}+A_h(t_0))^{-1}P_h)W_0dz\right \|_{L^2(\Omega)}\leq Ch^2t^{\alpha-1}\|W_0\|_{L^2(\Omega)}.
		\end{aligned}
	\end{equation*}
	If $W_0\in \dot{H}^2(\Omega)$, then one has
	\begin{equation*}
		\begin{aligned}
			&\left \|(E(t,t_0)-E_h(t,t_0)R_h)W_0\right \|_{L^2(\Omega)}\\
			\leq&\left\| \int_{\Gamma_{\theta,\kappa}}e^{zt}z^{-\alpha}(A(t_0)(z^{\alpha}+A(t_0))^{-1}-A_h(t_0)(z^{\alpha}+A_h(t_0))^{-1}R_h)W_0dz\right \|_{L^2(\Omega)}+\left\| \int_{\Gamma_{\theta,\kappa}}e^{zt}z^{-\alpha}(\mathbf{I}-R_h)W_0dz\right \|_{L^2(\Omega)}\\
			\leq&\left\| \int_{\Gamma_{\theta,\kappa}}e^{zt}z^{-\alpha}((z^{\alpha}+A(t_0))^{-1}-(z^{\alpha}+A_h(t_0))^{-1}P_h)A(t_0)W_0dz\right \|_{L^2(\Omega)}+Ch^2t^{\alpha-1}\|W_0\|_{\dot{H}^2(\Omega)}\\\leq& Ch^2t^{\alpha-1}\|W_0\|_{\dot{H}^2(\Omega)},
		\end{aligned}
	\end{equation*}
	because of Lemma \ref{lemerror1}, the fact $A_hR_h=P_hA$ \cite{Bazhlekova2015}, and $(z^\alpha+A)^{-1}=z^{-\alpha}(\mathbf{I}-A(z^{\alpha}+A)^{-1})$; and here $\mathbf{I}$ is the identity operator. Thus we get
	\begin{equation*}
		\begin{aligned}
			\|\upsilon_{1,h}(t_0)\|_{L^2(\Omega)}\leq&\left \|(P_hE(t,t_0)-E(t,t_0))W_0\right \|_{L^2(\Omega)}\\
			&+\left \|(E(t,t_0)-E_h(t,t_0)P_h)W_0\right \|_{L^2(\Omega)}\leq Ch^2t^{-1}\|W_0\|_{L^2(\Omega)} ~{\rm for} ~W_0\in L^2(\Omega)
		\end{aligned}
	\end{equation*}
	and
	\begin{equation*}
		\begin{aligned}
			\|\upsilon_{1,h}(t_0)\|_{L^2(\Omega)}\leq&\left \|(P_hE(t,t_0)-E(t,t_0))W_0\right \|_{L^2(\Omega)}\\
			&+\left \|(E(t,t_0)-E_h(t,t_0)R_h)W_0\right \|_{L^2(\Omega)}\leq Ch^2t^{\alpha-1}\|W_0\|_{\dot{H}^2(\Omega)} ~｛{\rm for} ~W_0\in \dot{H}^2(\Omega).
		\end{aligned}
	\end{equation*}
	Taking $t_0=t$ and using the interpolation property \cite{Adams2003} lead to
	\begin{equation*}
		\|\upsilon_{1,h}(t)\|_{L^2(\Omega)}\leq Ch^2t^{\epsilon\alpha/2-1}\|W_0\|_{\dot{H}^\epsilon(\Omega)}.
	\end{equation*}
	Similarly, one has
	\begin{equation*}
		\begin{aligned}
			&\|\upsilon_{2,h}(t)\|_{L^2(\Omega)}\leq Ch^2\|f(0)\|_{L^2(\Omega)}+Ch^2\int_{0}^{t}\|f'(s)\|_{L^2(\Omega)}ds,
		\end{aligned}
	\end{equation*}
according to Lemmas \ref{lemprojection}, \ref{lemerror1}, and the convolution property $f(t)=f(0)+(1\ast f')(t)$. As for $\upsilon_{3,h}(t)$, Theorem \ref{thmregofG} gives
	\begin{equation*}
		\begin{aligned}
			\|\upsilon_{3,h}(t_0)\|_{L^2(\Omega)}\leq& Ch^2\int_{0}^{t_0}\|G(s)\|_{\dot{H}^2(\Omega)}ds
			\leq Ch^2\|W_0\|_{\dot{H}^{\epsilon}(\Omega)}+Ch^2\|f(0)\|_{L^2(\Omega)}+Ch^2\int_0^{t_0}\|f'(s)\|_{L^2(\Omega)}ds.
		\end{aligned}
	\end{equation*}
Combining Lemma \ref{lemestEFh},  \eqref{eqassuma}, and $A_hR_h=P_hA$ results in
	\begin{equation*}
		\begin{aligned}
			\|\upsilon_{4,h}(t_0)\|_{L^2(\Omega)}\leq &\left \|\int_0^{t_0}E_h(t_0-s,t_0)(A_h(t_0)-A_h(s))\upsilon_h(s)ds\right \|_{L^2(\Omega)}\\
			&+\left \|\int_0^{t_0}E_h(t_0-s,t_0)A_h(t_0)(1-a^2(t_0)/a^2(s))(R_h-\mathbf{I})G(s)ds\right \|_{L^2(\Omega)}\\
			\leq& C\int_0^{t_0}\|\upsilon_h(s)\|_{L^2(\Omega)}ds+Ch^2\int_0^{t_0}\|G(s)\|_{\dot{H}^2(\Omega)}ds.
		\end{aligned}
	\end{equation*}
	Thus Gr\"{on}wall's inequality and Theorem \ref{thmregofG} imply the desired result.
	
\end{proof}
\begin{lemma}\label{lemspaestIV}
		If $W_0\in \dot{H}^{\epsilon}(\Omega)$, $f(0)\in L^2(\Omega)$, and $\int_{0}^{t}\|f'(s)\|_{L^2(\Omega)}ds<C$, then there holds
		\begin{equation*}
				\|\uppercase\expandafter{\romannumeral4}(t)\|_{L^2(\Omega)}\leq Ch^2\left (\|W_0\|_{\dot{H}^{\epsilon}(\Omega)}+\|f(0)\|_{L^2(\Omega)}+\int_{0}^{t}\|f'(s)\|_{L^2(\Omega)}ds\right ).
		\end{equation*}
\end{lemma}
\begin{proof}
	According to \eqref{eqehsep}, one can divide $\uppercase\expandafter{\romannumeral4}(t_0)$ into two parts, i.e.,
	\begin{equation*}
		\begin{aligned}
			\|\uppercase\expandafter{\romannumeral4}(t_0)\|_{L^2(\Omega)}\leq&\left \|\int_0^{t_0}F_h(t_0-s,t_0)(A_h(t_0)(R_h-P_h)-A_h(s)(R_h-P_h))\,_0D^{1-\alpha}_sW(s)ds\right \|_{L^2(\Omega)}\\
			&+\left \|\int_0^{t_0}F_h(t_0-s,t_0)(A_h(t_0)-A_h(s))\,_0D^{1-\alpha}_se_h(s)ds\right \|_{L^2(\Omega)}
			\leq \uppercase\expandafter{\romannumeral4}_1(t_0)+
			\uppercase\expandafter{\romannumeral4}_2(t_0).	\end{aligned}
	\end{equation*}
	By assumption \eqref{eqassuma}, Lemma \ref{lemestEF}, and Theorem \ref{thmregofG}, one can derive
	\begin{equation*}
		\begin{aligned}
			\uppercase\expandafter{\romannumeral4}_1(t_0)\leq&\int_{0}^{t_0}\|F_h(t_0-s,t_0)A_h(t_0)(1-a^2(t_0)/a^2(s))(R_h-P_h)\,_0D^{1-\alpha}_sW(s)\|_{L^2(\Omega)}ds\\
			\leq&Ch^2\int_0^{t_0}(t_0-s)^{1-\alpha}\|\,_0D^{1-\alpha}_sW(s)\|_{\dot{H}^2(\Omega)}ds\\
			\leq& Ch^2\left (\|W_0\|_{\dot{H}^{\epsilon}(\Omega)}+\|f(0)\|_{L^2(\Omega)}+\int_{0}^{t_0}\|f'(s)\|_{L^2(\Omega)}ds\right ).
		\end{aligned}
	\end{equation*}
	According to Lemmas \ref{lemestEFh}, \ref{lemepslonest} and assumption
	\eqref{eqassuma}, there are
	\begin{equation*}
		\begin{aligned}
			\uppercase\expandafter{\romannumeral4}_2(t_0)\leq& \int_{0}^{t_0}\|A_h(t_0)F_h(t_0-s,t_0)\|\|1-a^2(t_0)/a^2(s)\|\|\,_0D^{1-\alpha}_se_h(s)\|_{L^2(\Omega)}ds\\
			\leq& \int_{0}^{t_0}(t-s)^{1-\alpha}\|\,_0D^{1-\alpha}_se_h(s)\|_{L^2(\Omega)}ds\\
			\leq& Ch^2\left (\|W_0\|_{\dot{H}^{\epsilon}(\Omega)}+\|f(0)\|_{L^2(\Omega)}+\int_{0}^{t_0}\|f'(s)\|_{L^2(\Omega)}ds\right ).
		\end{aligned}
	\end{equation*}
Then the desired result is obtained after taking $t_0=t$.
\end{proof}
Combining Theorem \ref{thmregofW}, Lemmas \ref{lemprojection}, \ref{lemspaestI}, \ref{lemspaestII}, \ref{lemspaestIII}, and \ref{lemspaestIV} leads to the error estimate of spacial semi-discrete scheme.
\begin{theorem}\label{thmsemier}
	Let $W(t)$ and $W_h(t)$ be the solutions of Eqs. \eqref{eqretosol} and \eqref{eqsemischeme},  respectively. If $W_0\in \dot{H}^{\epsilon}(\Omega)$, $f(0)\in L^2(\Omega)$, and $\int_{0}^{t}\|f'(s)\|_{L^2(\Omega)}ds<+\infty$, then there holds
	\begin{equation*}
		\|W(t)-W_h(t)\|_{L^2(\Omega)}\leq Ch^2\left (t^{-\alpha}\|W_0\|_{\dot{H}^{\epsilon}(\Omega)}+\|f(0)\|_{L^2(\Omega)}+\int_{0}^{t}\|f'(s)\|_{L^2(\Omega)}ds\right ).
	\end{equation*}
\end{theorem}

\noindent{\section{Temporal discretization and error analysis}}
In this section, we use backward Euler convolution quadrature to discretize the time fractional derivative and  perform the error analyses of the fully discrete scheme for homogeneous and inhomogeneous problems. First, let the time step size $\tau=T/L$, $L\in\mathbb{N}$, $t_i=i\tau$, $i=0,1,\ldots,L$ and $0=t_0<t_1<\cdots<t_L=T$. Taking $\delta_\tau(\zeta)=\frac{1-\zeta}{\tau}$ and using convolution quadrature for Eq. \eqref{eqsemischeme}, we have the fully discrete scheme for any fixed integer $m\in[0,L]$,
\begin{equation}\label{eqfullscheme}
	\left \{\begin{aligned}
		&\frac{W^{n}_{h}-W^{n-1}_{h}}{\tau}+A_h(t_m)\sum_{i=0}^{n-1}d^{1-\alpha}_{i}W^{n-i}_h=f_h^n+(A_h(t_m)-A_h(t_n))\sum_{i=0}^{n-1}d^{1-\alpha}_{i}W^{n-i}_h,\\
		&W^{0}_{h}=W_{0,h},
	\end{aligned}\right .
\end{equation}
where
\begin{equation*}
\sum_{i=0}^{\infty}d^{1-\alpha}_i\zeta^i=\delta_\tau(\zeta)^{1-\alpha},\quad 0<\alpha<1,
\end{equation*}
and $W^n_h$ denotes the numerical solution of \eqref{eqsemischeme} at $t=t_n$.
Multiplying $\zeta^n$ on both sides of \eqref{eqfullscheme} and summing $n$ from $1$ to $\infty$ result in
\begin{equation*}
	\begin{aligned}
	\sum_{n=1}^{\infty}\frac{W^{n}_{h}-W^{n-1}_{h}}{\tau}\zeta^n+\sum_{n=1}^{\infty}A_h(t_m)\sum_{i=0}^{n-1}d^{1-\alpha}_{i}W^{n-i}_{h}\zeta^n=\sum_{n=1}^{\infty}f_h^n\zeta^n+\sum_{n=1}^{\infty}(A_h(t_m)-A_h(t_n))\sum_{i=0}^{n-1}d^{1-\alpha}_{i}W^{n-i}_h\zeta^n;
	\end{aligned}
\end{equation*}
after simple calculations, there exists
\begin{equation*}
	\begin{aligned}
	\left (\delta_\tau(\zeta)+A_h(t_m)\delta_\tau(\zeta)^{1-\alpha}\right )\sum_{n=1}^{\infty}W^n_h\zeta^n=\sum_{n=1}^{\infty}f_h^n\zeta^n+\sum_{n=1}^{\infty}(A_h(t_m)-A_h(t_n))\sum_{i=0}^{n-1}d^{1-\alpha}_{i}W^{n-i}_h\zeta^n+\frac{\zeta}{\tau}W^0_h.
	\end{aligned}
\end{equation*}
Thus, choosing $\xi_\tau=e^{-\tau(\kappa+1)}$, one has
\begin{equation}\label{eqrepsWnh}
	\begin{aligned}
	W^n_h=&\frac{1}{2\pi\mathbf{i} }\int_{\zeta=|\xi_\tau|}\zeta^{-n-1}\delta_\tau(\zeta)^{\alpha-1}\left (\delta_\tau(\zeta)^{\alpha}+A_h(t_m)\right )^{-1}\sum_{j=1}^{\infty}f_h^j\zeta^jd\zeta\\
	&+\frac{1}{2\pi\mathbf{i} }\int_{\zeta=|\xi_\tau|}\zeta^{-n-1}\delta_\tau(\zeta)^{\alpha-1}\left (\delta_\tau(\zeta)^{\alpha}+A_h(t_m)\right )^{-1}\frac{\zeta}{\tau}W^0_hd\zeta\\	
	&+\frac{1}{2\pi\mathbf{i} }\int_{\zeta=|\xi_\tau|}\zeta^{-n-1}\delta_\tau(\zeta)^{\alpha-1}\left (\delta_\tau(\zeta)^{\alpha}+A_h(t_m)\right )^{-1}\left (\sum_{j=1}^{\infty}(A_h(t_m)-A_h(t_j))\sum_{i=0}^{j-1}d^{1-\alpha}_{i}W^{j-i}_h\zeta^j\right )d\zeta.
	\end{aligned}
\end{equation}

Before providing the error estimates, we recall the following Lemma.
\begin{lemma}[\cite{Gunzburger2018}]\label{Lemseriesest}
	Let $0<\alpha<1$ and $\theta \in\left(\frac{\pi}{2}, \operatorname{arccot}\left(-\frac{2}{\pi}\right)\right)$ be given,  where $arccot$ means the inverse function of $\cot$, and let $\rho \in (0,1)$ be fixed. Then, both $\delta_\tau(e^{-z\tau})$ and $(\delta_\tau(e^{-z\tau})+A)^{-1}$ are analytic with respect to $z$ in the region enclosed by $\Gamma^\tau_\rho=\{z=-\ln{\rho}/\tau+\mathbf{i}y:y\in\mathbb{R}~and~|y|\leq \pi/\tau\}$, $\Gamma^\tau_{\theta,\kappa}=\{z\in \mathbb{C}:\kappa\leq |z|\leq\frac{\pi}{\tau\sin(\theta)},|\arg z|=\theta\}\bigcup\{z\in \mathbb{C}:|z|=\kappa,|\arg z|\leq\theta\}$, and the two lines $\mathbb{R}\pm \mathbf{i}\pi/\tau$ whenever $0<\kappa \leq \min (1 / T,-\ln (\rho) / \tau)$. Furthermore, there are the estimates
	\begin{equation*}
	\begin{aligned}
	&\delta_{\tau}\left(e^{-z \tau}\right) \in \Sigma_{\theta}&\forall z \in \Gamma_{\theta, \kappa}^{\tau},\\
	&C_{0}|z| \leq\left|\delta_{\tau}\left(e^{-z\tau }\right)\right| \leq C_{1}|z|&\forall z \in \Gamma_{\theta, \kappa}^{\tau},\\
	&\left|\delta_{\tau}\left(e^{-z\tau }\right)-z\right| \leq C \tau|z|^{2}&\forall z \in \Gamma_{\theta, \kappa}^{\tau},\\
	&\left|\delta_{\tau}\left(e^{-z\tau }\right)^{\alpha}-z^{\alpha}\right| \leq C \tau|z|^{\alpha+1}&\forall z \in \Gamma_{\theta, \kappa}^{\tau},
	\end{aligned}
	\end{equation*}
	where the constants $C_0$, $C_1$ and $C$ are independent of $\tau$ and $\kappa\in (0,\min (1 / T,-\ln (\rho) / \tau)]$.
	
\end{lemma}

Below we provide the error estimates of the homogeneous and inhomogeneous problems separately.
\noindent{\subsection{Error estimate for the inhomogeneous problem}}
In this subsection, we consider the error estimate between $W_h(t_n)$ and $W^n_h$ which are the solutions of Eqs. \eqref{eqsemischeme} and \eqref{eqfullscheme} with the initial value $W_{0}=0$.
Denote $e^n_h=W_h(t_n)-W^n_h$. By \eqref{eqrepsWh} and \eqref{eqrepsWnh}, there exists
\begin{equation}\label{errordd}
	\|e^n_h\|_{L^2(\Omega)}\leq \uppercase\expandafter{\romannumeral1}+\uppercase\expandafter{\romannumeral2},
\end{equation}
where
\begin{equation*}
	\begin{aligned}
		\uppercase\expandafter{\romannumeral1}\leq&C\left\|\int_0^{t_n}F_h(t_n-s,t_m)f_h(s)ds\right.-\left.\frac{1}{2\pi\mathbf{i} }\int_{\zeta=|\xi_\tau|}\zeta^{-n-1}\delta_\tau(\zeta)^{\alpha-1}\left (\delta_\tau(\zeta)^{\alpha}+A_h(t_m)\right )^{-1}\sum_{j=1}^{\infty}f_h^j\zeta^jd\zeta\right\|_{L^2(\Omega)},\\
		\uppercase\expandafter{\romannumeral2}\leq&C\left\|\int_0^{t_n}F_h(t_n-s,t_m)(A_h(t_m)-A_h(s))\,_0D^{1-\alpha}_sW_h(s)ds\right.\\
		&-\left.\frac{1}{2\pi\mathbf{i} }\int_{\zeta=|\xi_\tau|}\zeta^{-n-1}\delta_\tau(\zeta)^{\alpha-1}\left (\delta_\tau(\zeta)^{\alpha}+A_h(t_m)\right )^{-1}\left (\sum_{j=1}^{\infty}(A_h(t_m)-A_h(t_j))\sum_{i=0}^{j-1}d^{1-\alpha}_{i}W^{j-i}_h\zeta^j\right )d\zeta\right\|_{L^2(\Omega)}.
	\end{aligned}
\end{equation*}

Like the proof in \cite{Jin2016,Lubich1996}, one can get  the following estimates of  $\uppercase\expandafter{\romannumeral1}$ and $\uppercase\expandafter{\romannumeral2}$ defined in \eqref{errordd}.
\begin{theorem}\label{thmimhomI}
		If $f(0)\in L^2(\Omega)$ and $\int_{0}^{t}\|f'(s)\|_{L^2(\Omega)}ds<\infty$, then there holds
	\begin{equation*}
		\begin{aligned}
		&\uppercase\expandafter{\romannumeral1}\leq C\tau \|f(0)\|_{L^2(\Omega)}+C\tau\int_0^t\|f'(s)\|_{L^2(\Omega)}ds.
		\end{aligned}
	\end{equation*}
\end{theorem}
As for $\uppercase\expandafter{\romannumeral2}$, we introduce
\begin{equation*}
	\tau\sum_{i=0}^{\infty}F^{i}_{\tau,m}\zeta^i=\delta_\tau(\zeta)^{\alpha-1}\left (\delta_\tau(\zeta)^{\alpha}+A_h(t_m)\right )^{-1},
\end{equation*}
where
\begin{equation*}
	F^n_{\tau,m}=\frac{1}{2\pi\tau \mathbf{i}}\int_{\zeta=|\xi_\tau|}\zeta^{-n-1}\delta_\tau(\zeta)^{\alpha-1}\left (\delta_\tau(\zeta)^{\alpha}+A_h(t_m)\right )^{-1}d\zeta
\end{equation*}
and $\xi_\tau=e^{-\tau(\kappa+1)}$. Taking $\zeta=e^{-z\tau}$ and deforming the contour $\Gamma^\tau=\{z=\kappa+1+\mathbf{i}y:y\in\mathbb{R}~{\rm and}~|y|\leq \pi/\tau\}$ to
$\Gamma^\tau_{\theta,\kappa}$, one has 
\begin{equation*}
F^n_{\tau,m}=\frac{1}{2\pi \mathbf{i}}\int_{\Gamma^\tau_{\theta,\kappa}}e^{zn\tau}\delta_\tau(e^{-z\tau})^{\alpha-1}\left (\delta_\tau(e^{-z\tau})^{\alpha}+A_h(t_m)\right )^{-1}dz,
\end{equation*}
and simple calculations lead to
\begin{equation}\label{eqFntmest1}
	\left \|A_h(t_m)F^n_{\tau,m}\right \|=\left \|\frac{1}{2\pi \mathbf{i}}\int_{\Gamma^\tau_{\theta,\kappa}}e^{zn\tau}A_h(t_m)\delta_\tau(e^{-z\tau})^{\alpha-1}\left (\delta_\tau(e^{-z\tau} )^{\alpha}+A_h(t_m)\right )^{-1}dz\right \|\leq C(t_m+\tau)^{-\alpha}.
\end{equation}
To get the estimates of $\uppercase\expandafter{\romannumeral2}$, we divide it into four parts, i.e.,
\begin{equation}\label{equimII}
	\begin{aligned}
		\uppercase\expandafter{\romannumeral2}\leq&C\left\|\int_0^{t_n}F_h(t_n-s,t_m)(A_h(t_m)-A_h(s))\,_0D^{1-\alpha}_sW_h(s)ds\right.\\
		&-\left.\tau\sum_{k=1}^{n}F^{n-k}_{\tau,m}\left ((A_h(t_m)-A_h(t_k))\sum_{i=0}^{k-1}d^{1-\alpha}_{i}W^{k-i}_h\right )\right\|_{L^2(\Omega)}\leq\sum_{k=1}^n(\uppercase\expandafter{\romannumeral2}_{1,k}+\uppercase\expandafter{\romannumeral2}_{2,k}+\uppercase\expandafter{\romannumeral2}_{3,k}+\uppercase\expandafter{\romannumeral2}_{4,k}),
	\end{aligned}
\end{equation}
where
\begin{equation*}
	\begin{aligned}
	\uppercase\expandafter{\romannumeral2}_{1,k}\leq&C\left\|\tau F^{n-k}_{\tau,m}(A_h(t_m)-A_h(t_k))\left (\sum_{i=0}^{k-1}d^{1-\alpha}_{i}W^{k-i}_h-\,_0D^{1-\alpha}_tW_h(t_{k})\right )\right\|_{L^2(\Omega)},\\
	\uppercase\expandafter{\romannumeral2}_{2,k}\leq&C\left\|\left (\int_{t_{k-1}}^{t_k}F_h(t_n-s,t_m)ds-\tau F^{n-k}_{\tau,m}\right)\left ((A_h(t_m)-A_h(t_k))\,_0D^{1-\alpha}_tW_h(t_{k})\right )\right\|_{L^2(\Omega)},\\
	\uppercase\expandafter{\romannumeral2}_{3,k}\leq&C\bigg\|\int_{t_{k-1}}^{t_k}F_h(t_n-s,t_m)\left ((A_h(t_m)-A_h(s))-(A_h(t_m)-A_h(t_k))\right )\,_0D^{1-\alpha}_tW_h(t_{k})ds\bigg\|_{L^2(\Omega)},\\
	\uppercase\expandafter{\romannumeral2}_{4,k}\leq&C\bigg\|\int_{t_{k-1}}^{t_k}F_h(t_n-s,t_m)(A_h(t_m)-A_h(s))\left (\,_0D^{1-\alpha}_tW_h(s)-\,_0D^{1-\alpha}_tW_h(t_{k})\right )ds\bigg\|_{L^2(\Omega)}.\\
	\end{aligned}
\end{equation*}
To get error estimates of $\uppercase\expandafter{\romannumeral2}$, the following estimates of $G_h$ defined in \eqref{eqrepsGh} are also needed. Similar to Theorem \ref{thmregofG}, we have the following results.
\begin{theorem}\label{thmregofGh}
	If $W_0\in \dot{H}^{\epsilon}(\Omega)$, $f(0)\in L^2(\Omega)$, and $\int_{0}^{t}\|f'(s)\|_{L^2(\Omega)}ds<\infty$, then $G_h(t)$ satisfies
	\begin{equation*}
	\|G_h(t)\|_{L^2(\Omega)}\leq Ct^{\alpha-1}\|W_0\|_{L^2(\Omega)}+C\|f(0)\|_{L^2(\Omega)}+C\int_{0}^{t}\|f'(s)\|_{L^2(\Omega)}ds
	\end{equation*}
	and
	\begin{equation*}
	\|G_h(t)\|_{\dot{H}^2(\Omega)}\leq Ct^{\epsilon\alpha/2-1}\|W_0\|_{\dot{H}^{\epsilon}(\Omega)}+C\|f(0)\|_{L^2(\Omega)}+C\int_{0}^{t}\|f'(s)\|_{L^2(\Omega)}ds.
	\end{equation*}
\end{theorem}
\begin{theorem}\label{thmHolderG}
	Let $G_h(t)=\!_0D^{1-\alpha}_tW_h(t)$. Assume $W_0=0$, $f(0)\in L^2(\Omega)$, and $f'(s)\in L^{\infty}(0,T,L^2(\Omega))$. There holds
	\begin{equation*}
	\begin{aligned}
	\left\|\frac{G_h(t)-G_h(t-\tau)}{\tau^{\gamma}}\right\|_{L^2(\Omega)}\leq& Ct^{\alpha-\gamma}\left (\|f(0)\|_{L^2(\Omega)}+\|f'(s)\|_{L^{\infty}(0,T,L^2(\Omega))}\right ),
	\end{aligned}
	\end{equation*}
	where $\gamma<1+\alpha$.
\end{theorem}
\begin{proof}
	According to \eqref{eqrepsGh}, one has
	\begin{equation*}
	\begin{aligned}
	\left\|\frac{G_h(t)-G_h(t-\tau)}{\tau^{\gamma}}\right\|_{L^2(\Omega)}\leq&\upsilon_1+\upsilon_2,
	\end{aligned}
	\end{equation*}
	where
	\begin{equation*}
	\begin{aligned}
	\upsilon_1=&\left\|\frac{\int_0^tE_h(t-s,t_0)f_h(s)ds-\int_0^{t-\tau}E_h(t-\tau-s,t_0)f_h(s)ds}{\tau^\gamma}\right\|_{L^2(\Omega)},\\
	\upsilon_2=&\left\|\frac{\int_0^tE_h(t-s,t_0)(A_h(t_0)-A_h(s))G_h(s)ds-\int_0^{t-\tau}E_h(t-\tau-s,t_0)(A_h(t_0)-A_h(s))G_h(s)ds}{\tau^\gamma}\right\|_{L^2(\Omega)}.\\
	\end{aligned}
	\end{equation*}
	As for $\upsilon_1$, we split it into	
	\begin{equation*}
	\begin{aligned}
	\upsilon_1\leq& C\left\|\frac{\int_{0}^{t}E_h(t-s,t_0)dsf_h(0)-\int_{0}^{t-\tau}E_h(t-\tau-s,t_0)dsf_h(0)}{\tau^\gamma}\right\|_{L^2(\Omega)}\\
	&+C\left\|\frac{\int_{0}^{t-\tau}\left (\int_0^{t-s}E_h(r,t_0)dr-\int_0^{t-\tau-s}E_h(r,t_0)dr\right )f_h'(s)ds}{\tau^\gamma}\right\|_{L^2(\Omega)}\\
	&+C\left\|\frac{\int_{t-\tau}^{t}\int_0^{t-s}E_h(r,t_0)drf_h'(s)ds}{\tau^\gamma}\right\|_{L^2(\Omega)}\leq \upsilon_{1,1}+\upsilon_{1,2}+\upsilon_{1,3}.
	\end{aligned}
	\end{equation*}
	Using the fact $\left |\frac{1-e^{-z\tau}}{\tau^\gamma}\right |\leq C|z|^{\gamma}$ on $\Gamma_{\theta,\kappa}$,
	there is
	\begin{equation*}
		\begin{aligned}
		\upsilon_{1,1}\leq& C\left\|\int_{\Gamma_{\theta,\kappa}}e^{z(t-s)}\frac{1-e^{-z\tau}}{\tau^{\gamma}}(z^\alpha+A_h(t_0))^{-1}z^{-1}dzf_h(0)\right\|_{L^2(\Omega)}\\
		\leq&C\int_{\Gamma_{\theta,\kappa}}|e^{z(t-s)}||z|^{\gamma-\alpha-1}|dz|\|f_h(0)\|_{L^2(\Omega)}\leq Ct^{\alpha-\gamma}\|f(0)\|_{L^2(\Omega)}.
		\end{aligned}
	\end{equation*}
	Similarly, we can bound $\upsilon_{1,2}$ by
	\begin{equation*}
	\begin{aligned}
	\upsilon_{1,2}\leq& C\left\|\int_{0}^{t-\tau}\int_{\Gamma_{\theta,\kappa}}e^{z(t-s)}\frac{1-e^{-z\tau}}{\tau^{\gamma}}(z^\alpha+A_h(t_0))^{-1}z^{-1}dzf_h'(s)ds\right\|_{L^2(\Omega)}\\
	\leq
	&C\int_{0}^{t-\tau}\int_{\Gamma_{\theta,\kappa}}|e^{z(t-s)}||z|^{\gamma-\alpha-1}|dz|\|f_h'(s)\|_{L^2(\Omega)}ds
	\leq C\int_{0}^{t-\tau}(t-s)^{\alpha-\gamma}\|f'(s)\|_{L^2(\Omega)}ds,
	\end{aligned}
	\end{equation*}
where $\gamma<1+\alpha$ is required to ensure $\upsilon_{1,2}$ convergent. Similarly one has
	\begin{equation*}
	\begin{aligned}
	\upsilon_{1,3}\leq& C\left\|\int_{t-\tau}^{t}\int_{\Gamma_{\theta,\kappa}}e^{z(t-s)}\tau^{-\gamma}(z^\alpha+A_h(t_0))^{-1}z^{-1}dzf_h'(s)ds\right\|_{L^2(\Omega)}\\
	\leq& C\left\|\int_{\Gamma_{\theta,\kappa}}e^{z\tau}\frac{1-e^{-z\tau}}{z\tau^{\gamma}}(z^\alpha+A_h(t_0))^{-1}z^{-1}dz\right\|_{L^2(\Omega)}\|f_h'(s)\|_{L^{\infty}(0,T,L^2(\Omega))}\\
	\leq
	&C\int_{\Gamma_{\theta,\kappa}}|z|^{\gamma-\alpha-2}|dz|\|f_h'(s)\|_{L^{\infty}(0,T,L^2(\Omega))}
	\leq Ct^{1+\alpha-\gamma}\|f'(s)\|_{L^{\infty}(0,T,L^2(\Omega))},
	\end{aligned}
	\end{equation*}
	where we take $\kappa=1/t$ and require $\gamma<1+\alpha$ to ensure $\upsilon_{1,3}$ convergent. Thus there exist
	\begin{equation*}
	\upsilon_1\leq Ct^{\alpha-\gamma}\left (\|f(0)\|_{L^2(\Omega)}+\|f'(s)\|_{L^{\infty}(0,T,L^2(\Omega))}\right )
	\end{equation*}
	and $\gamma<1+\alpha$. Similarly, when $t=t_0$, one can split $\upsilon_{2}$ into
	\begin{equation*}
		\begin{aligned}
		\upsilon_2
		\leq&\left\|\frac{\int_0^{t_0-\tau}\left (E_h(t_0-s,t_0)-E_h(t_0-\tau-s,t_0)\right )(A_h(t_0)-A_h(s))G_h(s)ds}{\tau^\gamma}\right\|_{L^2(\Omega)}\\
		&+\left\|\frac{\int_{t_0-\tau}^{t_0}E_h(t_0-s,t_0)(A_h(t_0)-A_h(s))G_h(s)ds}{\tau^\gamma}\right\|_{L^2(\Omega)}\leq \upsilon_{2,1}+\upsilon_{2,2}.
		\end{aligned}
	\end{equation*}
	Using Lemma \ref{lemestEFh} and assumption \eqref{eqassumae}, one has
	\begin{equation*}
		\begin{aligned}
			\upsilon_{2,1}\leq& \left\|\int_0^{t_0-\tau}\int_{\Gamma_{\theta,\kappa}}e^{z(t_0-s)}\frac{1-e^{-z\tau}}{\tau^\gamma}(z^{\alpha}+A_h(t_0))^{-1}dz(A_h(t_0)-A_h(s))G_h(s)ds\right\|_{L^2(\Omega)}\\
			\leq& \int_0^{t_0-\tau}\int_{\Gamma_{\theta,\kappa}}|e^{z(t_0-s)}||z|^{\gamma-\alpha}|dz|\|(A_h(t_0)-A_h(s))G_h(s)\|_{L^2(\Omega)}ds
			\leq\int_0^{t_0-\tau}(t_0-s)^{\alpha-\gamma}\|G_h(s)\|_{\dot{H}^2(\Omega)}ds.
		\end{aligned}
	\end{equation*}
	Similarly, when $\gamma\leq 1$, there exists
	\begin{equation*}
		\upsilon_{2,2}\leq C\left\|\frac{\int_{t_0-\tau}^{t_0}E_h(t_0-s,t_0)ds}{\tau^{\gamma-1}}\right\|\|G_h(s)\|_{L^\infty(0,T,\dot{H}^2(\Omega))}\leq Ct_0^{1+\alpha-\gamma}\|G_h(s)\|_{L^\infty(0,T,\dot{H}^2(\Omega))};
	\end{equation*}
	when $\gamma>1$,
	\begin{equation*}
		\begin{aligned}
			\upsilon_{2,2}\leq& C\left\|\frac{\int_{t_0-\tau}^{t_0}E_h(t_0-s,t_0)ds}{\tau^{\gamma-1}}\right\|\|G_h(s)\|_{L^\infty(0,T,\dot{H}^2(\Omega))}\\
			\leq& C\left\|\frac{\int_{t_0-\tau}^{t_0}\int_{\Gamma_{\theta,\kappa}}e^{z(t_0-s)}(z^{\alpha}+A_h(t_0))^{-1}dzds}{\tau^{\gamma-1}}\right\|\|G_h(s)\|_{L^\infty(0,T,\dot{H}^2(\Omega))}\\
			\leq&C \left\|\int_{\Gamma_{\theta,\kappa}}\frac{1-e^{z\tau}}{z\tau^{\gamma-1}}(z^{\alpha}+A_h(t_0))^{-1}dz\right\|\|G_h(s)\|_{L^\infty(0,T,\dot{H}^2(\Omega))}\\
			\leq&C \int_{\Gamma_{\theta,\kappa}}|z|^{\gamma-2-\alpha}|dz|\|G_h(s)\|_{L^\infty(0,T,\dot{H}^2(\Omega))}\leq Ct_0^{1+\alpha-\gamma}\|G_h(s)\|_{L^\infty(0,T,\dot{H}^2(\Omega))},
		\end{aligned}
	\end{equation*}
	where we take $\kappa=1/t_0$ and require $\gamma<1+\alpha$ for the integral $\int_{\Gamma_{\theta,\kappa}}|z|^{\gamma-2-\alpha}|dz|$ to be convergent. Taking $t_0=t$ and using Theorem \ref{thmregofGh}, the desired results can be obtained by the fact $T/t>1$.
\end{proof}

Now we estimate $\uppercase\expandafter{\romannumeral2}$. First, we consider $\uppercase\expandafter{\romannumeral2}_{1,k}$ defined in \eqref{equimII}, which implies the difference between   $\sum_{i=0}^{k-1}d^{1-\alpha}_{i}W^{k-i}_h$ and $\,_0D^{1-\alpha}_tW_h(t_{k})$ needs to be obtained.
\begin{lemma}\label{lemerrorGhGnh}
	Let $G_h(t)=\!_0D^{1-\alpha}_tW_h(t)$ and $G^n_h=\sum_{i=0}^{n-1}d^{1-\alpha}_iW^{n-i}_h$. Assume $f(0)\in L^2(\Omega)$ and $f'(s)\in L^{\infty}(0,T,L^2(\Omega))$. There exists
	\begin{equation*}
	\|G^n_h-G_h(t_n)\|_{L^2(\Omega)}\leq C\tau t_n^{\alpha-1}\left (\|f(0)\|_{L^2(\Omega)}+\|f'(s)\|_{L^{\infty}(0,T,L^2(\Omega))}\right ).
	\end{equation*}
\end{lemma}
\begin{proof}
From the definition of $G^n_h$ and Eq. \eqref{eqrepsWnh}, it has
	\begin{equation*}
		\begin{aligned}
			&\sum_{n=1}^{\infty}G^{n}_h\zeta^n=\sum_{n=1}^{\infty}\sum_{i=0}^{n-1}d^{1-\alpha}_iW^{n-i}_h\zeta^n=\delta_\tau(\zeta)^{1-\alpha}\sum_{n=1}^{\infty}W^{n}_h\zeta^n\\
			=&\left (\delta_\tau(\zeta)^{\alpha}+A_h(t_m)\right)^{-1}\sum_{n=1}^{\infty}f_h^n\zeta^n+\left (\delta_\tau(\zeta)^{\alpha}+A_h(t_m)\right)^{-1}\sum_{n=1}^{\infty}(A_h(t_m)-A_h(t_n))G^n_h\zeta^n.
		\end{aligned}
	\end{equation*}
	Considering the error between $G^m_h$ and $G_h(t_m)$, one has
	\begin{equation*}
		\begin{aligned}
		\|G^m_h-G_h(t_m)\|_{L^2(\Omega)}\leq \sum_{k=1}^{2}\upsilon_{k,h}
		\end{aligned}
	\end{equation*}
	where
	\begin{equation*}
		\begin{aligned}
			\upsilon_{1,h}\leq&C\left \|\int_{\zeta=|\xi_\tau|}\zeta^{-m-1}\left (\delta_\tau(\zeta)^{\alpha}+A_h(t_m)\right)^{-1}\sum_{n=1}^{\infty}f_h^n\zeta^nd\zeta-\int_{0}^{t_m}E_h(t_m-s,t_m)f_h(s)ds\right \|_{L^2(\Omega)},\\
			\upsilon_{2,h}\leq&C\left \|\int_{\zeta=|\xi_\tau|}\zeta^{-m-1}\left (\delta_\tau(\zeta)^{\alpha}+A_h(t_m)\right)^{-1}\sum_{n=1}^{\infty}(A_h(t_m)-A_h(t_n))G^n_h\zeta^nd\zeta\right .\\&\left .-\int_{0}^{t_m}E_h(t_m-s,t_m)(A_h(t_m)-A_h(s))G_h(s)ds\right \|_{L^2(\Omega)}\\
		\end{aligned}
	\end{equation*}
	with $\xi_\tau=e^{-\tau(\kappa+1)}$. Similar to the proof in \cite{Jin2016,Lubich1996}, the following estimate of $\upsilon_{1,h}$ can be got
	\begin{equation*}
		\begin{aligned}
			&\upsilon_{1,h}\leq C\tau t_m^{\alpha-1}\|f(0)\|_{L^2(\Omega)
			}+C\tau\int_0^{t_m}(t_m-s)^{\alpha-1}\|f'(s)\|_{L^2(\Omega)}ds.\\
		\end{aligned}
	\end{equation*}
	As for $\upsilon_{2,h}$, we introduce
	\begin{equation*}
	\tau\sum_{i=0}^{\infty}E^{i}_{\tau,m}\zeta^i=\left (\delta_\tau(\zeta)^{\alpha}+A_h(t_m)\right )^{-1},
	\end{equation*}
	where
	\begin{equation*}
	E^n_{\tau,m}=\frac{1}{2\pi \mathbf{i}}\int_{\Gamma^\tau_{\theta,\kappa}}e^{zn\tau}\left (\delta_\tau(e^{-z\tau})^{\alpha}+A_h(t_m)\right )^{-1}dz.
	\end{equation*}
	Thus
	\begin{equation}\label{eqestAEn}
		\|A_h(t_m)E^n_{\tau,m}\|=\left \|\frac{1}{2\pi \mathbf{i}}\int_{\Gamma^\tau_{\theta,\kappa}}e^{zn\tau}A_h(t_m)\left (\delta_\tau(e^{-z\tau} )^{\alpha}+A_h(t_m)\right )^{-1}dz\right \|\leq C(t_m+\tau)^{-1}.
	\end{equation}
	For convenience, we split $\upsilon_{2,h}$ into the following forms
	\begin{equation*}
		\begin{aligned}
			\upsilon_{2,h}\leq&C\left \|\tau\sum_{k=1}^{m}E^{m-k}_{\tau,m}(A_h(t_m)-A_h(t_k))G^k_h-\int_{0}^{t_m}E_h(t_m-s,t_m)(A_h(t_m)-A_h(s))G_h(s)ds\right \|_{L^2(\Omega)}\\
			\leq& C\sum_{k=1}^{m} \left \|\tau E^{m-k}_{\tau,m}(A_h(t_m)-A_h(t_k))(G^k_h-G_h(t_k))\right \|_{L^2(\Omega)}\\
			&+C\sum_{k=1}^{m} \left \|\left (\tau E^{m-k}_{\tau,m}-\int_{t_{k-1}}^{t_k}E_h(t_m-s,t_m)ds\right)(A_h(t_m)-A_h(t_k))G_h(t_k)\right \|_{L^2(\Omega)}\\
			&+C\sum_{k=1}^{m} \left \|\int_{t_{k-1}}^{t_k}E_h(t_m-s,t_m)(A_h(t_k)-A_h(s))dsG_h(t_k)\right \|_{L^2(\Omega)}\\
			&+C\sum_{k=1}^{m} \left \|\int_{t_{k-1}}^{t_k}E_h(t_m-s,t_m)(A_h(t_m)-A_h(s))(G_h(t_k)-G_h(s))ds\right \|_{L^2(\Omega)}=\sum_{k=1}^{m }\sum_{j=1}^{4}\upsilon_{2,j,k,h}.
		\end{aligned}
	\end{equation*}
	Assumption \eqref{eqassuma} and Eq. \eqref{eqestAEn} lead to
	\begin{equation*}
		\begin{aligned}
			\sum_{k=1}^{m}\upsilon_{2,1,k,h}\leq C\sum_{k=1}^{m}\tau\|G^k_h-G_h(t_k)\|_{L^2(\Omega)}.
		\end{aligned}
	\end{equation*}
	As for $\upsilon_{2,2,k,h}$, there is
	\begin{equation*}
		\begin{aligned}
		&\left \|\tau E^{m-k}_{\tau,m
		}-\int_{t_{k-1}}^{t_k}E_h(t_m-s,t_m)ds\right \|\leq \left \| \int_{t_{k-1}}^{t_k}E^{m-k}_{\tau,m
	}-E_h(t_m-s,t_m)ds\right \|\\
\leq&C\left \| \int_{t_{k-1}}^{t_k}\int_{\Gamma^\tau_{\theta,\kappa}}e^{z(m-k)\tau}\left (\delta_\tau(e^{-z\tau})^{\alpha}+A_h(t_m)\right )^{-1}dz-\int_{\Gamma_{\theta,\kappa}}e^{z(t_m-s)}(z^{\alpha}+A_h(t_m))^{-1}dzds\right \|\\
\leq&C\left \| \int_{t_{k-1}}^{t_k}\int_{\Gamma_{\theta,\kappa}\backslash\Gamma^\tau_{\theta,\kappa}}e^{z(t_m-s)}(z^{\alpha}+A_h(t_m))^{-1}dzds\right \|\\
&+C\left \| \int_{t_{k-1}}^{t_k}\int_{\Gamma^\tau_{\theta,\kappa}}e^{z(t_m-s)}(1-e^{(s-k\tau)z})\left (\delta_\tau(e^{-z\tau} )^{\alpha}+A_h(t_m)\right )^{-1}dzds\right \|\\
&+C\left \| \int_{t_{k-1}}^{t_k}\int_{\Gamma^\tau_{\theta,\kappa}}e^{z(t_m-s)}\left (\left (\delta_\tau(e^{-z\tau})^{\alpha}+A_h(t_m)\right )^{-1}-(z^{\alpha}+A_h(t_m))^{-1}\right )dzds\right \|\\
\leq &C\tau\int_{t_{k-1}}^{t_k}(t_m-s)^{\alpha-2}ds.
		\end{aligned}
	\end{equation*}
According to \eqref{eqassumae}, it has
\begin{equation*}
\begin{aligned}
	\sum_{k=1}^{m}\upsilon_{2,2,k,h}\leq C\tau\sum_{k=1}^{m}\int_{t_{k-1}}^{t_k}(t_m-s)^{\alpha-1}ds\|G_h(t_k)\|_{\dot{H}^2(\Omega)}
	\leq C\tau\|f(0)\|_{L^2(\Omega)}+C\tau\int_{0}^{t_m}\|f'(s)\|_{L^2(\Omega)}ds,
\end{aligned}
\end{equation*}
where the last inequality follows by Lemma \ref{thmregofG}. Combining Lemma \ref{lemestEF} and assumption \eqref{eqassumae}, there exists
	\begin{equation*}
		\begin{aligned}
			\sum_{k=1}^{m}\upsilon_{2,3,k,h}\leq C\sum_{k=1}^{m}\tau\int_{t_{k-1}}^{t_k}(t_m-s)^{\alpha-1} \|G_h(t_k)\|_{\dot{H}^2(\Omega)}ds
				\leq C\tau\|f(0)\|_{L^2(\Omega)}+C\tau\int_{0}^{t_m}\|f'(s)\|_{L^2(\Omega)}ds.
		\end{aligned}
	\end{equation*}
From Lemma \ref{thmHolderG}, it holds
	\begin{equation*}
	\begin{aligned}
	\sum_{k=1}^{m}\upsilon_{2,4,k,h}\leq& C\tau\sum_{k=1}^{m}\int_{t_{k-1}}^{t_k}(t_m-s)^{\alpha} \left \|\frac{G_h(s)-G_h(t_k)}{\tau}\right \|_{L^2(\Omega)}ds
	\leq C\tau\left (\|f(0)\|_{L^2(\Omega)}+\|f'(s)\|_{L^{\infty}(0,T,L^2(\Omega))}\right ).
	\end{aligned}
	\end{equation*}
	Since $m\in[0,L]$ is any fixed integer, taking $m=n$ results in
	\begin{equation*}
		\|G^n_h-G_h(t_n)\|_{L^2(\Omega)}\leq C\tau t_n^{\alpha-1}\left (\|f(0)\|_{L^2(\Omega)}+\|f'(s)\|_{L^{\infty}(0,T,L^2(\Omega))}\right )+C\sum_{k=1}^{n}\tau\|G^k_h-G_h(t_k)\|_{L^2(\Omega)}.
	\end{equation*}
	Then the discrete Gr\"{o}nwall inequality \cite{Thomee2006} leads to the desired results.
\end{proof}

\begin{theorem}\label{thmimhomII}
	If $f(0)\in L^2(\Omega)$ and $f'(s)\in L^{\infty}(0,T,L^2(\Omega))$, then there holds
	\begin{equation*}
	\uppercase\expandafter{\romannumeral2}\leq C\tau\left (\|f(0)\|_{L^2(\Omega)}+\|f'(s)\|_{L^{\infty}(0,T,L^2(\Omega))}\right ),
	\end{equation*}
	where $\uppercase\expandafter{\romannumeral2}$  is defined in \eqref{equimII}.
\end{theorem}

\begin{proof}
	According to Lemma \ref{lemerrorGhGnh} and Eq. \eqref{eqFntmest1}, it has
	\begin{equation*}
		\begin{aligned}
			\sum_{k=1}^m\uppercase\expandafter{\romannumeral2}_{1,k}\leq& C\tau \sum_{k=1}^m(t_m-t_k)^{1-\alpha}\|G^k_h-G_h(t_k)\|_{L^2(\Omega)}\\
\leq& C\tau \|f(0)\|_{L^2(\Omega)}+C\tau\|f'(s)\|_{L^{\infty}(0,T,L^2(\Omega))}.
		\end{aligned}
	\end{equation*}
	Next, consider the difference between $\tau F^{m-k}_{\tau,m}$ and $\int_{t_{k-1}}^{t_k}F_h(t_m-s,t_m)ds$, i.e.,
	\begin{equation*}
	\begin{aligned}
	&\left \|\tau F^{m-k}_{\tau,m
	}-\int_{t_{k-1}}^{t_k}F_h(t_m-s,t_m)ds\right \|\leq \left \| \int_{t_{k-1}}^{t_k}F^{m-k}_{\tau,m
	}-F_h(t_m-s,t_m)ds\right \|\\
	\leq&C\left \| \int_{t_{k-1}}^{t_k}\int_{\Gamma^\tau_{\theta,\kappa}}e^{z(m-k)\tau}\delta_\tau(e^{-z\tau})^{\alpha-1}\left (\delta_\tau(e^{-z\tau})^{\alpha}+A_h(t_m)\right )^{-1}dz\right.\\
	&-\left .\int_{\Gamma_{\theta,\kappa}}e^{(t_m-s)z}z^{\alpha-1}(z^{\alpha}+A_h(t_m))^{-1}dzds\right \|\\
	\leq&C\left \| \int_{t_{k-1}}^{t_k}\int_{\Gamma_{\theta,\kappa}\backslash\Gamma^\tau_{\theta,\kappa}}e^{(t_m-s)z}z^{\alpha-1}(z^{\alpha}+A_h(t_m))^{-1}dzds\right \|\\
	&+C\left \| \int_{t_{k-1}}^{t_k}\int_{\Gamma^\tau_{\theta,\kappa}}e^{z(t_m-s)}(1-e^{(s-k\tau)z})\delta_\tau(e^{-z\tau})^{\alpha-1}\left (\delta_\tau(e^{-z\tau})^{\alpha}+A_h(t_m)\right )^{-1}dzds\right \|\\
	&+C\left \| \int_{t_{k-1}}^{t_k}\int_{\Gamma^\tau_{\theta,\kappa}}e^{z(t_m-s)}\left (\delta_\tau(e^{-z\tau})^{\alpha-1}\left (\delta_\tau(e^{-z\tau} )^{\alpha}+A_h(t_m)\right )^{-1}-z^{\alpha-1}(z^{\alpha}+A_h(t_m))^{-1}\right )dzds\right \|\\
	\leq &C\tau\int_{t_{k-1}}^{t_k}(t_m-s)^{-1}ds,
	\end{aligned}
	\end{equation*}
	where Lemma \ref{Lemseriesest} is used. According to assumption \eqref{eqassumae} and $t_m-t_k\leq t_m-s$ for $s\in[t_{k-1},t_k]$, one has
	\begin{equation*}
	\begin{aligned}
	\sum_{k=1}^m\uppercase\expandafter{\romannumeral2}_{2,k}\leq C\tau\sum_{k=1}^m\|G_h(t_k)\|_{\dot{H}^2(\Omega)}\leq C\tau\|f(0)\|_{L^2(\Omega)}+C\tau\int_{0}^{t_m}\|f'(s)\|_{L^2(\Omega)}ds.
	\end{aligned}
	\end{equation*}
	Using Lemma \ref{lemestEFh}, assumption \eqref{eqassumae}, and Theorem \ref{thmregofGh}, one can get
	\begin{equation*}
	\begin{aligned}
	\sum_{k=1}^m\uppercase\expandafter{\romannumeral2}_{3,k}\leq C\tau\sum_{k=1}^m\int_{t_{k-1}}^{t_k}\|G_h(s)\|_{\dot{H}^2(\Omega)}ds\leq C\tau\|f(0)\|_{L^2(\Omega)}+C\tau\int_{0}^{t_m}\|f'(s)\|_{L^2(\Omega)}ds.
	\end{aligned}
	\end{equation*}
Combining Lemma \ref{lemestEFh}, assumption \eqref{eqassumae}, and Theorem \ref{thmHolderG} results in
		\begin{equation*}
	\begin{aligned}
	\sum_{k=1}^m\uppercase\expandafter{\romannumeral2}_{4,k}\leq C\tau\sum_{k=1}^m\int_{t_{k-1}}^{t_k}(t_m-s)^{1-\alpha}\left \|\frac{G_h(s)-G_h(t_k)}{\tau}\right \|_{L^2(\Omega)}ds\leq C\tau\|f(0)\|_{L^2(\Omega)}+C\tau\|f'(s)\|_{L^{\infty}(0,T,L^2(\Omega))}.
	\end{aligned}
	\end{equation*}
	Thus taking $m=n$ leads to
	\begin{equation*}
		\uppercase\expandafter{\romannumeral2}\leq C\tau\left (\|f(0)\|_{L^2(\Omega)}+\|f'(s)\|_{L^{\infty}(0,T,L^2(\Omega))}\right ).
	\end{equation*}
\end{proof}

Combining Theorems \ref{thmimhomI} and \ref{thmimhomII}, one gets the result.
\begin{theorem}\label{thmfullinhomo}
	Let $W_h$ and $W^n_h$ be the solutions of Eqs. \eqref{eqsemischeme} and \eqref{eqfullscheme} respectively. If $f(0)\in L^2(\Omega)$ and $f'(s)\in L^{\infty}(0,T,L^2(\Omega))$, then there holds
	\begin{equation*}
	\|W_h(t_n)-W^n_h\|_{L^2(\Omega)}\leq C\tau\left (\|f(0)\|_{L^2(\Omega)}+\|f'(s)\|_{L^{\infty}(0,T,L^2(\Omega))}\right ).
	\end{equation*}
\end{theorem}

\subsection{Error estimate for the homogeneous problem}
In this subsection, we consider the error between $W_h(t_n)$ and $W^n_h$, which are the solutions of Eqs. \eqref{eqsemischeme} and \eqref{eqfullscheme} with $f=0$. Similarly, denote $e^n_h=W_h(t_n)-W^n_h$. Thus
\begin{equation*}
	\|e^n_h\|_{L^2(\Omega)}\leq \uppercase\expandafter{\romannumeral1}+\uppercase\expandafter{\romannumeral2},
\end{equation*}
where
\begin{equation*}
	\begin{aligned}
		\uppercase\expandafter{\romannumeral1}\leq&C\left\|F_h(t_n,t_m)W_{0,h}-\frac{1}{2\pi\mathbf{i} }\int_{\zeta=|\xi_\tau|}\zeta^{-n-1}\delta_\tau(\zeta)^{\alpha-1}\left (\delta_\tau(\zeta)^{\alpha}+A_h(t_m)\right )^{-1}\frac{\zeta}{\tau}W^0_hd\zeta\right\|_{L^2(\Omega)},\\
		\uppercase\expandafter{\romannumeral2}\leq&C\left\|\int_0^{t_n}F_h(t_n-s,t_m)(A_h(t_m)-A_h(s))\,_0D^{1-\alpha}_sW_h(s)ds\right.\\
		&-\left.\frac{1}{2\pi\mathbf{i} }\int_{\zeta=|\xi_\tau|}\zeta^{-n-1}\delta_\tau(\zeta)^{\alpha-1}\left (\delta_\tau(\zeta)^{\alpha}+A_h(t_m)\right )^{-1}\left (\sum_{j=1}^{\infty}(A_h(t_m)-A_h(t_j))\sum_{i=0}^{j-1}d^{1-\alpha}_{i}W^{j-i}_h\zeta^j\right )d\zeta\right\|_{L^2(\Omega)}
	\end{aligned}
\end{equation*}
with $\xi_\tau=e^{-\tau(\kappa+1)}$. Similar to the proof in \cite{Jin2016,Lubich1996}, one can obtain the following estimates.

\begin{theorem}\label{thmhomI}
	If $W_0\in \dot{H}^\epsilon(\Omega)$ with $\epsilon>0$, then there exists
	\begin{equation*}
		\begin{aligned}
			&\uppercase\expandafter{\romannumeral1}\leq Ct^{\alpha\epsilon/2-1}\tau \|W_0\|_{\dot{H}^{\epsilon}(\Omega)}.
		\end{aligned}
	\end{equation*}
\end{theorem}

As for $\uppercase\expandafter{\romannumeral2}$, when $t_n=t_m$, it has

\begin{equation}\label{eqromannumera}
	\begin{aligned}
		\uppercase\expandafter{\romannumeral2}\leq&C\left\|\int_0^{t_m}F_h(t_m-s,t_m)(A_h(t_m)-A_h(s))\,_0D^{1-\alpha}_sW_h(s)ds\right.\\
		&-\left.\tau\sum_{k=1}^{m}F^{m-k}_{\tau,m}\left ((A_h(t_m)-A_h(t_k))\sum_{i=0}^{k-1}d^{1-\alpha}_{i}W^{k-i}_h\right )\right\|_{L^2(\Omega)}\\
		\leq&C\left\|-\int_0^{t_m}\frac{\partial}{\partial s}\left (F_h(t_m-s,t_m)(A_h(t_m)-A_h(s))\right )\,_0I^{\alpha}_sW_h(s)ds\right.\\
		&-\left.\tau\sum_{k=1}^{m}\left (F^{m-k}_{\tau,m}(A_h(t_m)-A_h(t_k))\right )\left (\sum_{i=0}^{k-1}d^{-\alpha}_{i}W^{k-i}_h-\sum_{i=0}^{k-2}d^{-\alpha}_{i}W^{k-1-i}_h\right )/\tau\right\|_{L^2(\Omega)}\\
		\leq&C\left\|\int_0^{t_m}\frac{\partial}{\partial (t_m-s)}\left (F_h(t_m-s,t_m)(A_h(t_m)-A_h(s))\right )\,_0I^{\alpha}_sW_h(s)ds\right.\\
		&-\left.\tau\sum_{k=1}^{m}\left (F^{m-k}_{\tau,m}(A_h(t_m)-A_h(t_k))-F^{m-k-1}_{\tau,m}(A_h(t_m)-A_h(t_{k-1}))\right )/\tau\sum_{i=0}^{k-1}d^{-\alpha}_{i}W^{k-i}_h\right\|_{L^2(\Omega)}\\
		\leq&\sum_{k=1}^{m}(\uppercase\expandafter{\romannumeral2}_{1,k}+\uppercase\expandafter{\romannumeral2}_{2,k}+\uppercase\expandafter{\romannumeral2}_{3,k}),
	\end{aligned}
\end{equation}
where
\begin{equation*}
	\begin{aligned}
		\uppercase\expandafter{\romannumeral2}_{1,k}\leq& \left\|\left (F^{m-k}_{\tau,m}(A_h(t_m)-A_h(t_k))-F^{m-k-1}_{\tau,m}(A_h(t_m)-A_h(t_{k-1}))\right )\left (\!_0I^{\alpha}_{t_k}W_h-\sum_{i=0}^{k-1}d^{-\alpha}_{i}W^{k-i}_h\right )\right\|_{L^2(\Omega)},\\
		\uppercase\expandafter{\romannumeral2}_{2,k}\leq& \left\|\bigg(\int_{t_{k-1}}^{t_k}\frac{\partial}{\partial (t_m-s)}\left (F_h(t_m-s,t_m)(A_h(t_m)-A_h(s))\right )ds\right .\\
		&\left .-\left (F^{m-k}_{\tau,m}(A_h(t_m)-A_h(t_k))-F^{m-k-1}_{\tau,m}(A_h(t_m)-A_h(t_{k-1}))\right )\bigg) \!_0I^{\alpha}_{t_k}W_h\right\|_{L^2(\Omega)},\\
		\uppercase\expandafter{\romannumeral2}_{3,k}\leq& \left\|\int_{t_{k-1}}^{t_k}\frac{\partial}{\partial (t_m-s)}\left (F_h(t_m-s,t_m)(A_h(t_m)-A_h(s))\right )(\!_0I^{\alpha}_{s}W_h-\!_0I^{\alpha}_{t_k}W_h)ds\right\|_{L^2(\Omega)}.\\
	\end{aligned}
\end{equation*}

To estimate $\uppercase\expandafter{\romannumeral2}_{1,k}$, denote $U_h(t_k)=\!_0I^{\alpha}_{t_k}W_h$ and $U^k_h=\sum_{i=0}^{k-1}d^{-\alpha}_{i}W^{k-i}_h$. By means of Laplace transform, $U_h$ can be represented by
\begin{equation*}
	U_h(t)=H_h(t,t_m)W^{0}_{h}+\int_0^t H_h(t-s,t_m)(A_h(t_m)-A_h(s))\frac{\partial}{\partial s}U_h(s)ds
\end{equation*}
and
\begin{equation*}
	\begin{aligned}
		U^{n}_h=&\sum_{i=0}^{n-1}d^{-\alpha}_{i}W^{n-i}_h
		=\frac{1}{2\pi \mathbf{i}}\int_{\zeta=|\xi_\tau|}\zeta^{-n-1}\delta_\tau(\zeta)^{-1}(\delta_\tau(\zeta)^{\alpha}+A_h(t_m))^{-1}\frac{\zeta}{\tau}W^0_hd\zeta\\
		&+\frac{1}{2\pi \mathbf{i}}\int_{\zeta=|\xi_\tau|}\zeta^{-n-1}\delta_\tau(\zeta)^{-1}(\delta_\tau(\zeta)^{\alpha}+A_h(t_m))^{-1}\sum_{j=1}^{\infty}(A_h(t_m)-A_h(t_j))(U^j_h-U^{j-1}_h)/\tau\zeta^jd\zeta,
	\end{aligned}
\end{equation*}
where
\begin{equation*}
	H_h(t,t_m)=\frac{1}{2\pi \mathbf{i}}\int_{\Gamma_{\theta,\kappa}}e^{zt}z^{-1}(z^\alpha+A_h(t_m))^{-1}dz.
\end{equation*}
Similar to the proof of Theorems \ref{thmregofG} and \ref{thmHolderG}, one can get the following estimates of $U_h(t)$.
\begin{theorem}\label{thmregofUh}
	If $W_0\in L^2(\Omega)$, then  $U_h(t)$ satisfies
	\begin{equation*}
	\|U_h(t)\|_{L^2(\Omega)}\leq C\|W_0\|_{L^2(\Omega)},\qquad
	\|U_h(t)\|_{\dot{H}^2(\Omega)}\leq C\|W_0\|_{L^2(\Omega)}.
	\end{equation*}
And if $W_0\in \dot{H}^\eta(\Omega)$, $\eta\in[0,2]$, then it holds
\begin{equation*}
		\|A_h^{\eta/2} U_h(t)\|_{\dot{H}^2(\Omega)}\leq C\|W_0\|_{\dot{H}^{\eta}(\Omega)}.
\end{equation*}
\end{theorem}
\begin{theorem}\label{thmHolderU}
	Let $U_h=\!_0I^{\alpha}_tW_h(t)$. When $W_0\in L^2(\Omega)$, there exists
	\begin{equation*}
	\begin{aligned}
	\left\|\frac{U_h(t)-U_h(t-\tau)}{\tau^{\gamma_1}}\right\|_{L^2(\Omega)}\leq& Ct^{\alpha-\gamma_1}\|W_0\|_{L^2(\Omega)},
	\end{aligned}
	\end{equation*}
		where $\gamma_1<1+\alpha$. And when $W_0\in \dot{H}^\epsilon(\Omega)$, one has
	\begin{equation*}
	\begin{aligned}
	\left\|A_h\frac{U_h(t)-U_h(t-\tau)}{\tau^{\gamma_2}}\right\|_{L^2(\Omega)}\leq& Ct^{\alpha\epsilon/2-\gamma_2}\|W_0\|_{\dot{H}^\epsilon(\Omega)},
	\end{aligned}
	\end{equation*}
	where $\gamma_2\leq 1$.
\end{theorem}
Then we consider the difference between $U_h(t_k)$ and $U^k_h$.
\begin{theorem}
	If $W_0\in L^2(\Omega)$ and $\frac{1}{a^2(t)}\in C^2[0,T]$, then there exists
	\begin{equation*}
	\|U^n_h-U_h(t_n)\|_{L^2(\Omega)}\leq Ct_n^{\alpha-1}\tau\|W_0\|_{L^2(\Omega)}.
	\end{equation*}
\end{theorem}
\begin{proof}
Here, for $n=m$, we can split it into
\begin{equation*}
	\|U_h(t_m)-U^m_h\|_{L^2(\Omega)}\leq \upsilon_{1,h}+\upsilon_{2,h},
\end{equation*}
where
\begin{equation*}
	\begin{aligned}
		\upsilon_{1,h}\leq& C\left \|H_h(t_m,t_m)W_{0,h}-\frac{1}{2\pi \mathbf{i}}\int_{\zeta=|\xi_\tau|}\zeta^{-m-1}\delta_\tau(\zeta)^{-1}(\delta_\tau(\zeta)^{\alpha}+A_h(t_m))^{-1}\frac{\zeta}{\tau}W^0_hd\zeta\right \|_{L^2(\Omega)},\\
		\upsilon_{2,h}\leq& C\left \|\int_0^{t_m} H_h(t_m-s,t_m)(A_h(t_m)-A_h(s))\frac{\partial}{\partial s}U_h(s)ds\right .\\
		&\left .-\frac{1}{2\pi \mathbf{i}}\int_{\zeta=|\xi_\tau|}\zeta^{-m-1}\delta_\tau(\zeta)^{-1}(\delta_\tau(\zeta)^{\alpha}+A_h(t_m))^{-1}\sum_{j=1}^{\infty}(A_h(t_m)-A_h(t_j))(U^j_h-U^{j-1}_h)/\tau\zeta^jd\zeta\right \|_{L^2(\Omega)}.\\
	\end{aligned}
\end{equation*}
Similar to  the proof                                                                                                                                               in \cite{Jin2016,Lubich1996}, the following estimate can be got
	\begin{equation*}
\upsilon_{1,h}\leq C\tau t_m^{\alpha-1}\|W_0\|_{L^2(\Omega)}.
\end{equation*}
To get the estimate of $\upsilon_{2,h}$, introduce
\begin{equation*} \tau\sum_{i=0}^{\infty}H^{i}_{\tau,m}\zeta^i=\delta_\tau(\zeta)^{-1}\left (\delta_\tau(\zeta)^{\alpha}+A_h(t_m)\right )^{-1},
\end{equation*}
where
\begin{equation*}
	H^n_{\tau,m}=\frac{1}{2\pi \mathbf{i}}\int_{\Gamma^\tau_{\theta,\kappa}}e^{-zn\tau}\delta_\tau(e^{-z\tau})^{-1}\left (\delta_\tau(e^{-z\tau})^{\alpha}+A_h(t_m)\right )^{-1}dz.
\end{equation*}
The $\upsilon_{2,h}$ can be divided into the following parts, i.e.,
\begin{equation*}
	\begin{aligned}
		\upsilon_{2,h}\leq & C\left \|\int_0^{t_m} \frac{\partial}{\partial t_m-s}H_h(t_m-s,t_m)(A_h(t_0)-A_h(s))U_h(s)ds\right .\\
		&\left .-\sum_{j=1}^{m} (H^{m-j}_{\tau,m}-H^{m-j-1}_{\tau,m})(A_h(t_m)-A_h(t_j))U^j_h\right \|_{L^2(\Omega)}=\sum_{i=1}^3\sum_{k=1}^{m}\upsilon_{2,i,k,h},
	\end{aligned}
\end{equation*}
where
\begin{equation*}
	\begin{aligned}
		\upsilon_{2,1,k,h}\leq& \left\|\left (H^{m-k}_{\tau,m}(A_h(t_m)-A_h(t_k))-H^{m-k-1}_{\tau,m}(A_h(t_m)-A_h(t_{k-1}))\right )\left (U_h(t_k)-U^k_h\right )\right\|_{L^2(\Omega)},\\
		\upsilon_{2,2,k,h}\leq& \left\|\bigg(\int_{t_{k-1}}^{t_k}\frac{\partial}{\partial (t_m-s)}\left (H_h(t_m-s,t_m)(A_h(t_m)-A_h(s))\right )ds\right .\\
		&\left .-\left (H^{m-k}_{\tau,m}(A_h(t_m)-A_h(t_k))-H^{m-k-1}_{\tau,m}(A_h(t_m)-A_h(t_{k-1}))\right )\bigg)U_h(t_k)\right\|_{L^2(\Omega)},\\
		\upsilon_{2,3,k,h}\leq& \left\|\int_{t_{k-1}}^{t_k}\frac{\partial}{\partial (t_m-s)}\left (H_h(t_m-s,t_m)(A_h(t_m)-A_h(s))\right )(U_h(s)-U_h(t_k))ds\right\|_{L^2(\Omega)}.\\
	\end{aligned}
\end{equation*}
As for $\upsilon_{2,1,k,h}$, it has
\begin{equation*}
	\begin{aligned}
		&\|H^{m-k}_{\tau,m}(A_h(t_m)-A_h(t_k))-H^{m-k-1}_{\tau,m}(A_h(t_m)-A_h(t_{k-1}))\|_{L^2(\Omega)}\\
		\leq& \|H^{m-k}_{\tau,m}(A_h(t_m)-A_h(t_k))-H^{m-k-1}_{\tau,m}(A_h(t_m)-A_h(t_{k}))\|\\
		&+\|H^{m-k-1}_{\tau,m}(A_h(t_k)-A_h(t_{k-1}))\|\leq \sigma_{1,k}+\sigma_{2,k}.
	\end{aligned}
\end{equation*}
Using Lemma \ref{Lemseriesest}, one can obtain
\begin{equation*}
	\begin{aligned}
		\sigma_{1,k}\leq C\tau(t_m-t_k)\left \|\int_{\Gamma^\tau_{\theta,\kappa}}e^{(m-k)\tau z}\frac{1-e^{-z\tau}}{\tau}A_h(t_m)\delta_\tau(e^{-z\tau})^{-1}(\delta_\tau(e^{-z\tau})^{\alpha}+A_h(t_m))^{-1}dz\right \|
		\leq C\tau.
	\end{aligned}
\end{equation*}
Similarly,
\begin{equation*}
	\sigma_{2,k}\leq C\tau\left \|\int_{\Gamma^\tau_{\theta,\kappa}}e^{(m-k-1)\tau z}A_h(t_m)\delta_\tau(e^{-z\tau})^{-1}(\delta_\tau(e^{-z\tau})^{\alpha}+A_h(t_m))^{-1}dz\right \|\leq C\tau.
\end{equation*}
Thus
\begin{equation*}
	\begin{aligned}
		\sum_{k=1}^{m}\upsilon_{2,1,k,h}\leq C\tau\sum_{k=1}^{m}\|U_h(t_k)-U^k_h\|_{L^2(\Omega)}.
	\end{aligned}
\end{equation*}
As for $\upsilon_{2,2,k,h}$, one can divide it into four parts, i.e.,
\begin{equation*}
	\begin{aligned}
		\upsilon_{2,2,k,h}
		\leq&C\left\|\int_{t_{k-1}}^{t_k}(A_h(t_m)-A_h(s))\left( \frac{\partial}{\partial (t_m-s)}H_h(t_m-s,t_m)-\left (H^{m-k}_{\tau,m}-H^{m-k-1}_{\tau,m}\right)/\tau\right)dsU_h(t_k) \right\|_{L^2(\Omega)}\\
		&+C\left\|\int_{t_{k-1}}^{t_k}(A_h(s)-A_h(t_{k-1}))\left (H^{m-k}_{\tau,m}-H^{m-k-1}_{\tau,m}\right)/\tau dsU_h(t_k) \right\|_{L^2(\Omega)}\\
		&+C\left\|\int_{t_{k-1}}^{t_k}\frac{\partial}{\partial s}A_h(s)\left( H_h(t_m-s,t_m)-H^{m-k}_{\tau,m}\right)ds U_h(t_k)\right\|_{L^2(\Omega)}\\
		&+C\left\|\int_{t_{k-1}}^{t_k}\left (\frac{\partial}{\partial s}A_h(s)-\frac{A_h(t_k)-A_h(t_{k-1})}{\tau}\right )H^{m-k}_{\tau,m}dsU_h(t_k) \right\|_{L^2(\Omega)}\leq \sum_{i=1}^{4}\vartheta_{i,k}.
	\end{aligned}
\end{equation*}
For the first part $\vartheta_{1,k}$, using Lemma \ref{Lemseriesest}, one has the estimate of the difference between $\frac{\partial}{\partial (t_m-s)}H_h(t_m-s,t_m)$ and $\frac{H^{m-k}_{\tau,m}-H^{m-k-1}_{\tau,m}}{\tau}$, i.e.,
\begin{equation*}
	\begin{aligned}
		&\left \| \frac{\partial}{\partial (t_m-s)}H_h(t_m-s,t_m)-\frac{H^{m-k}_{\tau,m}-H^{m-k-1}_{\tau,m}}{\tau}\right \|\\
		\leq &C\left \| \int_{\Gamma_{\theta,\kappa}}e^{z(t_m-s)}(z^{\alpha}+A_h(t_m))^{-1}dz-\int_{\Gamma^\tau_{\theta,\kappa}}e^{z(t_m-t_{k})}(\delta_\tau(e^{-z\tau})^{\alpha}+A_h(t_m))^{-1}dz\right \|\\
		\leq &C\tau(t_m-s)^{\alpha-2},
	\end{aligned}
\end{equation*}
which yields
\begin{equation*}
	\begin{aligned}
		\sum_{k=1}^{m}\vartheta_{1,k}\leq& C\tau\sum_{k=1}^{m} \int_{t_{k-1}}^{t_k}(t_m-s)^{\alpha-1}\|U_h(t_k)\|_{\dot{H}^2(\Omega)}ds.
	\end{aligned}
\end{equation*}
Similarly, one has
\begin{equation*}
	\left \|\frac{H^{m-k}_{\tau,m}-H^{m-k-1}_{\tau,m}}{\tau}\right \|\leq C\left \|\int_{\Gamma^\tau_{\theta,\kappa}}e^{z(t_m-t_{k-1})}e^{-z\tau}(\delta_\tau(e^{-z\tau})^{\alpha}+A_h(t_m))^{-1}dz\right \|\leq C(t_m-s)^{\alpha-1}.
\end{equation*}
Therefore one has
\begin{equation*}
	\begin{aligned}
		\sum_{k=1}^{m}\vartheta_{2,k}\leq& C\tau\sum_{k=1}^{m} \int_{t_{k-1}}^{t_k}(t_m-s)^{\alpha-1}\|U_h(t_k)\|_{\dot{H}^2(\Omega)}ds.
	\end{aligned}
\end{equation*}
Moreover, according to Lemma \ref{Lemseriesest}, there is
\begin{equation*}
	\begin{aligned}
		&\left \| H_h(t_m-s,t_m)-H^{m-k}_{\tau,m}\right \|\\
		\leq &C\left \| \int_{\Gamma_{\theta,\kappa}}e^{z(t_m-s)}z^{-1}(z^{\alpha}+A_h(t_m))^{-1}dz-\int_{\Gamma^\tau_{\theta,\kappa}}e^{z(t_m-t_{k})}\delta_\tau(e^{-z\tau})^{-1}(\delta_\tau(e^{-z\tau})^{\alpha}+A_h(t_m))^{-1}dz\right \|\\
		\leq &C\tau(t_m-s)^{\alpha-1},
	\end{aligned}
\end{equation*}
which leads to
\begin{equation*}
	\begin{aligned}
		\sum_{k=1}^{m}\vartheta_{3,k}\leq& C\tau\sum_{k=1}^{m} \int_{t_{k-1}}^{t_k}(t_m-s)^{\alpha-1}\|U_h(t_k)\|_{\dot{H}^2(\Omega)}ds.
	\end{aligned}
\end{equation*}
On the other hand, according to Lemma \ref{Lemseriesest}, one has
\begin{equation*}
	\begin{aligned}
		&\left \|H^{m-k}_{\tau,m}\right \|
		\leq C\left \| \int_{\Gamma^\tau_{\theta,\kappa}}e^{z(t_m-t_{k})}\delta_\tau(e^{-z\tau})^{-1}(\delta_\tau(e^{-z\tau})^{\alpha}+A_h(t_m))^{-1}dz\right \|
		\leq C(t_m-t_{k})^{\alpha}.
	\end{aligned}
\end{equation*}
Combining $\frac{1}{a^2(t)}\in C^2[0,T]$ leads to
\begin{equation}\label{eqqSigma}
	\begin{aligned}
		\sum_{k=1}^{m}\vartheta_{4,k}\leq& C\tau\sum_{k=1}^{m} \int_{t_{k-1}}^{t_k}(t_m-t_{k})^{\alpha}\|U_h(t_k)\|_{\dot{H}^2(\Omega)}ds.
	\end{aligned}
\end{equation}
The estimate \eqref{eqqSigma} together with Theorem \ref{thmregofUh} yield that
\begin{equation*}
	\begin{aligned}
		\sum_{k=1}^{m}\upsilon_{2,2,k,h}\leq C\tau\sum_{k=1}^{m} \int_{t_{k-1}}^{t_k}(t_m-s)^{\alpha-1}\|U_h(t_k)\|_{\dot{H}^2(\Omega)}ds\leq C\tau\|W_0\|_{L^2(\Omega)}.
	\end{aligned}
\end{equation*}
Using the condition $\frac{1}{a^2(t)}\in C^2[0,T]$, one can bound $\upsilon_{2,3,k,h}$ by
\begin{equation*}
	\begin{aligned}
	\upsilon_{2,3,k,h}\leq&C\left\|\int_{t_{k-1}}^{t_k}(A_h(t_m)-A_h(s))\frac{\partial}{\partial (t_m-s)}\left(H_h(t_m-s,t_m)\right )(U_h(s)-U_h(t_k))ds\right\|_{L^2(\Omega)}\\
	&+C\left\|\int_{t_{k-1}}^{t_k}H_h(t_m-s,t_m)\frac{\partial}{\partial s}\left(A_h(t_m)-A_h(s)\right )(U_h(s)-U_h(t_k))ds\right\|_{L^2(\Omega)}\\
	\leq& C\tau\int_{t_{k-1}}^{t_k}\left \|\frac{U_h(s)-U_h(t_k)}{\tau}\right \|_{L^2(\Omega)}ds.
	\end{aligned}
\end{equation*}
According to Theorem \ref{thmHolderU}, one has
\begin{equation*}
	\sum_{k=1}^{m}\upsilon_{2,3,k,h}\leq C\tau\|W_0\|_{L^2(\Omega)}.
\end{equation*}
Thus using discrete Gr\"{o}nwall inequality and taking $m=n$ result in

\begin{equation*}
	\|U^n_h-U_h(t_n)\|_{L^2(\Omega)}\leq Ct_n^{\alpha-1}\tau\|W_0\|_{L^2(\Omega)}.
\end{equation*}
\end{proof}

Next consider the estimate of $\uppercase\expandafter{\romannumeral2}$ defined in \eqref{eqromannumera}.
\begin{theorem}
If $W_0\in L^2(\Omega)$ and $\frac{1}{a^2(t)}\in C^2[0,T]$, then there exists
	\begin{equation*}
	\begin{aligned}
	\sum_{k=1}^{m}\uppercase\expandafter{\romannumeral2}_{1,k}\leq C\tau \|W_0\|_{L^2(\Omega)},
	\end{aligned}
	\end{equation*}
where $\uppercase\expandafter{\romannumeral2}_{1,k}$ is defined in \eqref{eqromannumera}.
\end{theorem}
\begin{proof}
By triangle inequality, we can divide it into two parts, i.e.,
\begin{equation*}
	\begin{aligned}
		&\|F^{m-k}_{\tau,m}(A_h(t_m)-A_h(t_k))-F^{m-k-1}_{\tau,m}(A_h(t_m)-A_h(t_{k-1}))\|\\
		\leq& \|F^{m-k}_{\tau,m}(A_h(t_m)-A_h(t_k))-F^{m-k-1}_{\tau,m}(A_h(t_m)-A_h(t_{k}))\|\\
		&+\|F^{m-k-1}_{\tau,m}(A_h(t_k)-A_h(t_{k-1}))\|\leq \varrho_{1,k}+\varrho_{2,k}.
	\end{aligned}
\end{equation*}
The fact $|\frac{1-e^{-z\tau}}{\tau}|\leq C|z|$ and Lemma \ref{Lemseriesest} show
\begin{equation*}
	\begin{aligned}
		 \varrho_{1,k}\leq C\tau(t_m-t_k)\left \|\int_{\Gamma^\tau_{\theta,\kappa}}e^{(m-k)\tau z}\frac{e^{-z\tau}-1}{\tau}A_h(t_m)\delta_\tau(e^{-z\tau})^{\alpha-1}(\delta_\tau(e^{-z\tau})^{\alpha}+A_h(t_m))^{-1}dz\right \|
		\leq C(t_m-s)^{-\alpha}\tau.
	\end{aligned}
\end{equation*}
Similarly
\begin{equation*}
	 \varrho_{2,k}\leq C\tau\left \|\int_{\Gamma^\tau_{\theta,\kappa}}e^{(m-k-1)\tau z}A_h(t_m)\delta_\tau(e^{-z\tau})^{\alpha-1}(\delta_\tau(e^{-z\tau})^{\alpha}+A_h(t_m))^{-1}dz\right \|\leq C(t_m-s)^{-\alpha}\tau.
\end{equation*}
Thus
\begin{equation*}
	\sum_{k=1}^{m}\uppercase\expandafter{\romannumeral2}_{1,k}\leq C\tau\|W_0\|_{L^2(\Omega)}.
\end{equation*}
\end{proof}
\begin{theorem}
If $W_0\in \dot{H}^\epsilon(\Omega)$ and $\frac{1}{a^2(t)}\in C^2[0,T]$, then there exists
	\begin{equation*}
	\begin{aligned}
	\sum_{k=1}^{m}\uppercase\expandafter{\romannumeral2}_{2,k}\leq C\tau \|W_0\|_{\dot{H}^{\epsilon}(\Omega)},
	\end{aligned}
	\end{equation*}
where $\uppercase\expandafter{\romannumeral2}_{2,k}$ is defined in \eqref{eqromannumera}.
\end{theorem}
\begin{proof}
By triangle inequality, there exists
\begin{equation*}
	\begin{aligned}
		\uppercase\expandafter{\romannumeral2}_{2,k}\leq&C\left\|\bigg(\int_{t_{k-1}}^{t_k}\frac{\partial}{\partial (t_m-s)}\left (F_h(t_m-s,t_m)(A_h(t_m)-A_h(s))\right )ds\right .\\
		&\left .-\left (F^{m-k}_{\tau,m}(A_h(t_m)-A_h(t_k))-F^{m-k-1}_{\tau,m}(A_h(t_m)-A_h(t_{k-1}))\right )\bigg)U_h(t_k)\right\|_{L^2(\Omega)}\\
		\leq&C\left\|\int_{t_{k-1}}^{t_k}(A_h(t_m)-A_h(s))\left( \frac{\partial}{\partial (t_m-s)}F_h(t_m-s,t_m)-\left (F^{m-k}_{\tau,m}-F^{m-k-1}_{\tau,m}\right)/\tau\right)dsU_h(t_k) \right\|_{L^2(\Omega)}\\
		&+C\left\|\int_{t_{k-1}}^{t_k}(A_h(s)-A_h(t_{k-1}))\left (F^{m-k}_{\tau,m}-F^{m-k-1}_{\tau,m}\right)/\tau dsU_h(t_k) \right\|_{L^2(\Omega)}\\
		&+C\left\|\int_{t_{k-1}}^{t_k}\frac{\partial}{\partial s}A_h(s)\left( F_h(t_m-s,t_m)-F^{m-k}_{\tau,m}\right)ds U_h(t_k)\right\|_{L^2(\Omega)}\\
		&+C\left\|\int_{t_{k-1}}^{t_k}\left (\frac{\partial}{\partial s}A_h(s)-\frac{A_h(t_k)-A_h(t_{k-1})}{\tau}\right )F^{m-k}_{\tau,m}dsU_h(t_k) \right\|_{L^2(\Omega)}\leq \sum_{i=1}^{4}\ell_{i,k}.
	\end{aligned}
\end{equation*}
From Lemma \ref{lemestEFh}, one has
\begin{equation*}
	\begin{aligned}
		&\left \| A_h(t_m)^{-\epsilon/2}\frac{\partial}{\partial (t_m-s)}F_h(t_m-s,t_m)-A_h(t_m)^{-\epsilon/2}\frac{F^{m-k}_{\tau,m}-F^{m-k-1}_{\tau,m}}{\tau}\right \|\\
		\leq &C\left \| \int_{\Gamma_{\theta,\kappa}}e^{z(t_m-s)}A_h(t_m)^{1-\epsilon/2}(z^{\alpha}+A_h(t_m))^{-1}dz-\int_{\Gamma^\tau_{\theta,\kappa}}e^{z(t_m-t_{k})}A_h(t_m)^{1-\epsilon/2}(\delta_\tau(e^{-z\tau})^{\alpha}+A_h(t_m))^{-1}dz\right \|\\
		&+C\left \| A_h(t_m)^{-\epsilon/2}\left (\int_{\Gamma_{\theta,\kappa}}e^{z(t_m-s)}\mathbf{I}dz-\int_{\Gamma^\tau_{\theta,\kappa}}e^{z(t_m-t_{k})}\mathbf{I}dz\right )\right \|\\
		\leq&C\left \| \int_{\Gamma_{\theta,\kappa}\backslash\Gamma^\tau_{\theta,\kappa}}e^{z(t_m-s)}A_h(t_m)^{1-\epsilon/2}(z^{\alpha}+A_h(t_m))^{-1}dz\right \|\\
		&+C\left \|\int_{\Gamma^\tau_{\theta,\kappa}}e^{z(t_m-s)}(1-e^{z(s-t_k)})A_h(t_m)^{1-\epsilon/2}(\delta_\tau(e^{-z\tau})^{\alpha}+A_h(t_m))^{-1}dz\right \|\\
		&+C\left \|\int_{\Gamma^\tau_{\theta,\kappa}}e^{z(t_m-s)}A_h(t_m)^{1-\epsilon/2}((z^{\alpha}+A_h(t_m))^{-1}-(\delta_\tau(e^{-z\tau})^{\alpha}+A_h(t_m))^{-1})dz\right \|\\
		\leq &C\tau(t_m-s)^{\alpha\epsilon/2-2},
	\end{aligned}
\end{equation*}
which implies
\begin{equation*}
	\begin{aligned}
		\sum_{k=1}^{m}\ell_{1,k}\leq& C\tau\sum_{k=1}^{m} \int_{t_{k-1}}^{t_k}(t_m-s)^{\alpha\epsilon/2-1}\|A_h(t_m)^{\epsilon/2}U_h(t_k)\|_{\dot{H}^2(\Omega)}ds.
	\end{aligned}
\end{equation*}
Similarly,
\begin{equation*}
\begin{aligned}
	\left \|A_h(t_m)^{-\epsilon/2}\frac{F^{m-k}_{\tau,m}-F^{m-k-1}_{\tau,m}}{\tau}\right \|&\leq \left \|\int_{\Gamma^\tau_{\theta,\kappa}}e^{z(t_m-t_{k-1})}e^{-z\tau}A_h(t_m)^{-\epsilon/2}\delta_\tau(e^{-z\tau})^{\alpha}(\delta_\tau(e^{-z\tau})^{\alpha}+A_h(t_m))^{-1}dz\right \|\\
	&\leq C(t_m-s)^{\alpha\epsilon/2-1}.
\end{aligned}
\end{equation*}
Therefore $\sum_{k=1}^{m}\ell_{2,k}$ can be bounded as
\begin{equation*}
	\begin{aligned}
		\sum_{k=1}^{m}\ell_{2,k}\leq& C\tau\sum_{k=1}^{m} \int_{t_{k-1}}^{t_k}(t_m-s)^{\alpha\epsilon/2-1}\|A_h(t_m)^{\epsilon/2}U_h(t_k)\|_{\dot{H}^2(\Omega)}ds.
	\end{aligned}
\end{equation*}
On the other hand, one can get
\begin{equation*}
	\begin{aligned}
		&\left \| A_h(t_m)^{-\epsilon/2}F_h(t_m-s,t_m)-A_h(t_m)^{-\epsilon/2}F^{m-k}_{\tau,m}\right \|\\
		\leq &C\left \| \int_{\Gamma_{\theta,\kappa}}e^{z(t_m-s)}A_h(t_m)^{1-\epsilon/2}z^{-1}(z^{\alpha}+A_h(t_m))^{-1}dz\right .\\
		&\left .-\int_{\Gamma^\tau_{\theta,\kappa}}e^{z(t_m-t_{k})}A_h(t_m)^{1-\epsilon/2}\delta_\tau(e^{-z\tau})^{-1}(\delta_\tau(e^{-z\tau})^{\alpha}+A_h(t_m))^{-1}dz\right \|\\
		&+C\left \|A_h(t_m)^{-\epsilon/2}\left ( \int_{\Gamma_{\theta,\kappa}}e^{z(t_m-s)}z^{-1}dz\right .\left .-\int_{\Gamma^\tau_{\theta,\kappa}}e^{z(t_m-t_{k})}\delta_\tau(e^{-z\tau})^{-1}dz\right )\right \|\\
		\leq &C\tau(t_m-s)^{\alpha\epsilon/2-1},
	\end{aligned}
\end{equation*}
which leads to
\begin{equation*}
	\begin{aligned}
		\sum_{k=1}^{m}\ell_{3,k}\leq& C\tau\sum_{k=1}^{m} \int_{t_{k-1}}^{t_k}(t_m-s)^{\alpha\epsilon/2-1}\|A_h(t_m)^{\epsilon/2}U_h(t_k)\|_{\dot{H}^2(\Omega)}ds.
	\end{aligned}
\end{equation*}
Next, using
\begin{equation*}
	\begin{aligned}
		&\left \|F^{m-k}_{\tau,m}\right \|
		\leq C\left \| \int_{\Gamma^\tau_{\theta,\kappa}}e^{z(t_m-t_{k})}\delta_\tau(e^{-z\tau})^{\alpha-1}(\delta_\tau(e^{-z\tau})^{\alpha}+A_h(t_m))^{-1}dz\right \|
		\leq C,
	\end{aligned}
\end{equation*}
there exists
\begin{equation*}
	\begin{aligned}
		\sum_{k=1}^{m}\ell_{4,k}\leq& C\tau\sum_{k=1}^{m} \int_{t_{k-1}}^{t_k}\|U_h(t_k)\|_{\dot{H}^2(\Omega)}ds.
	\end{aligned}
\end{equation*}
Thus, by Lemma \ref{thmregofUh}, it has
\begin{equation*}
	\begin{aligned}
		\sum_{k=1}^{m}\uppercase\expandafter{\romannumeral2}_{2,k}\leq C\tau\sum_{k=1}^{m} \int_{t_{k-1}}^{t_k}(t_m-s)^{\alpha\epsilon/2-1}\|A_h(t_m)^{\epsilon/2}U_h(t_k)\|_{\dot{H}^2(\Omega)}ds\leq C\tau \|W_0\|_{\dot{H}^{\epsilon}(\Omega)}.
	\end{aligned}
\end{equation*}
\end{proof}
\begin{theorem}
If $W_0\in \dot{H}^\epsilon(\Omega)$ and $\frac{1}{a^2(t)}\in C^2[0,T]$, then there exists
	\begin{equation*}
	\begin{aligned}
	\sum_{k=1}^{m}\uppercase\expandafter{\romannumeral2}_{3,k}\leq C\tau \|W_0\|_{\dot{H}^{\epsilon}(\Omega)},
	\end{aligned}
	\end{equation*}
where $\uppercase\expandafter{\romannumeral2}_{3,k}$ is defined in \eqref{eqromannumera}.
\end{theorem}

\begin{proof}
According to Theorem \ref{thmHolderU},  $\uppercase\expandafter{\romannumeral2}_{3,k}$ can be bounded as
	\begin{equation*}
		\begin{aligned}
			\uppercase\expandafter{\romannumeral2}_{3,k}\leq& \left\|\int_{t_{k-1}}^{t_k}(A_h(t_m)-A_h(s))\frac{\partial}{\partial (t_m-s)}\left (F_h(t_m-s,t_m)\right )(U_h(s)-U_h(t_k))ds\right \|_{L^2(\Omega)}\\
		&+\left\|\int_{t_{k-1}}^{t_k}F_h(t_m-s,t_m)\frac{\partial}{\partial (t_m-s)}\left((A_h(t_m)-A_h(s))\right)(U_h(s)-U_h(t_k))ds\right \|_{L^2(\Omega)}\\
			\leq& C\tau	\int_{t_{k-1}}^{t_k}\left\|A_h(t_m)\frac{U_h(s)-U_h(t_k)}{\tau}\right \|_{L^2(\Omega)}ds\\
			\leq& C\tau\int_{t_{k-1}}^{t_k}	s^{\alpha\epsilon/2-1}\|W_0\|_{\dot{H}^{\epsilon}(\Omega)}ds.
		\end{aligned}
	\end{equation*}
Summing $k$ from $1$ to $n$ leads to the desired estimate.
\end{proof}
Thus the error estimate of the fully discrete scheme when $f=0$ is obtained.
\begin{theorem}\label{thmfullhom}
	Let $W_h$ and $W^n_h$ be the solutions of Eqs. \eqref{eqsemischeme} and \eqref{eqfullscheme} respectively. If $\frac{1}{a^2(t)}\in C^2[0,T]$, $W_0\in \dot{H}^\epsilon(\Omega)$, and $f=0$, then there holds
	\begin{equation*}
	\|W_h(t_n)-W^n_h\|\leq C\tau t_n^{\alpha\epsilon-1}\|W_0\|_{\dot{H}^\epsilon(\Omega)}.
	\end{equation*}
\end{theorem}

\section{Numerical experiments}
In this section, we perform three numerical experiments with unknown explicit solution to verify the effectiveness of the designed schemes. 
The spatial errors will be measured by
\begin{equation*}
\begin{aligned}
E_{h}=\|W^{n}_{h}-W^{n}_{h/2}\|_{L^2(\Omega)},
\end{aligned}
\end{equation*}
where $W^n_{h}$ means the numerical solution of $W$ at time $t_n$ with mesh size $h$; similarly, we measure the temporal errors by
\begin{equation*}
\begin{aligned}
E_{\tau}=\|W^n_{\tau}-W^n_{\tau/2}\|_{L^2(\Omega)},
\end{aligned}
\end{equation*}
where $W^n_{\tau}$ are the numerical solutions of $W$ at the fixed time $t_n$ with step size $\tau$. The  corresponding convergence rates can be calculated by
\begin{equation*}
{\rm Rate}=\frac{\ln(E_{h}/E_{h/2})}{\ln(2)} ~{\rm and }~ {\rm Rate}=\frac{\ln(E_{\tau}/E_{\tau/2})}{\ln(2)}.
\end{equation*}
For convenience, we take $\Omega=(0,1)$.

\begin{example}
Here we consider temporal convergence rates for inhomogeneous problem \eqref{eqretosol}. Let
\begin{equation*}
  \frac{1}{a^2(t)}=t^{1.01},\quad
  f(x,t)=t^{0.1}\chi_{[0,1/2]},
\end{equation*}
and $T=1$, where $\chi_{[a,b]}$ denotes the characteristic function on $[a,b]$. To investigate the convergence in time and eliminate the influence from spatial discretization, we set $h=1/128$ and Table \ref{tab:timeu00} shows the errors and convergence rates when $\alpha=0.3$ and $0.7$. The results shown in Table \ref{tab:timeu00} validate Theorem \ref{thmfullinhomo}.
\begin{table}[htbp]
	\centering
	\caption{Temporal errors and convergence rates for inhomogeneous problem}
	\begin{tabular}{c|ccccc}
		\hline
		$\alpha\backslash \tau$& 1/50  & 1/100 & 1/200 & 1/400 & 1/800 \\
		\hline
		0.3 & 7.038E-04 & 3.269E-04 & 1.506E-04 & 6.899E-05 & 3.150E-05 \\
		& Rate     & 1.1063  & 1.1186  & 1.1259  & 1.1310  \\
		\hline
		0.7& 2.661E-04 & 1.225E-04 & 5.646E-05 & 2.601E-05 & 1.197E-05 \\
		& Rate    & 1.1186  & 1.1180  & 1.1183  & 1.1192  \\
		\hline
	\end{tabular}%
	\label{tab:timeu00}%
\end{table}%
\end{example}

\begin{example}
Here we validate temporal convergence rates for homogeneous problem \eqref{eqretosol}. To satisfy the condition provided in Theorem \ref{thmfullhom}, we take
\begin{equation*}
  \frac{1}{a^2(t)}=t^{2.01}.
\end{equation*}
Set $T=1$ and
\begin{equation*}
  W_0(x)=\chi_{(1/2,1]}.
\end{equation*}
We take small spatial mesh
size $h=1/128$ so that the spatial discretization error is relatively negligible. The corresponding results are shown in Table \ref{tab:timef00}, which agree with the predictions of Theorem \ref{thmfullhom}.
\begin{table}[htbp]
	\centering
	\caption{Temporal errors and convergence rates for homogeneous problem}
	\begin{tabular}{c|ccccc}
		\hline
		$\alpha\backslash \tau$& 1/50  & 1/100 & 1/200 & 1/400 & 1/800 \\
		\hline
		0.4   & 8.319E-03 & 4.193E-03 & 1.997E-03 & 9.421E-04 & 4.534E-04 \\
		& Rate  & 0.9885  & 1.0705  & 1.0835  & 1.0552  \\
		\hline
		0.6   & 3.802E-03 & 1.873E-03 & 9.194E-04 & 4.542E-04 & 2.256E-04 \\
		& Rate  & 1.0217  & 1.0262  & 1.0172  & 1.0095  \\
		\hline
	\end{tabular}%
	\label{tab:timef00}%
\end{table}%
\end{example}

\begin{example}
Finally, we take
\begin{equation*}
  \frac{1}{a^2(t)}=10t^{1.01},\quad  W_0(x)=\chi_{(1/2,1]},\quad f(x,t)=t^{0.1}\chi_{[0,1/2]}
\end{equation*}
to verify the spatial convergence rates.
Here we choose $T=2$ and $\tau=1/1000$. Table \ref{tab:spac} shows the errors and convergence rates,  which agree with the predictions of Theorem \ref{thmsemier}.
\begin{table}[htbp]
	\centering
	\caption{Spatial errors and convergence rates}
	\begin{tabular}{c|ccccc}
		\hline
		$\alpha\backslash h$	& 1/32  & 1/64  & 1/128 & 1/256 & 1/512 \\
		\hline
		0.2   & 9.828E-04 & 2.483E-04 & 6.224E-05 & 1.557E-05 & 3.893E-06 \\
		& Rate  & 1.9848  & 1.9962  & 1.9990  & 1.9998  \\
		\hline
		0.7   & 1.196E-04 & 3.341E-05 & 8.675E-06 & 2.192E-06 & 5.494E-07 \\
		& Rate  & 1.8395  & 1.9453  & 1.9849  & 1.9961  \\
		\hline
	\end{tabular}%
	\label{tab:spac}%
\end{table}%
\end{example}
\section{Conclusion}

The model describing anomalous diffusion in expanding media is with variable coefficient. The finite element method and backward Euler convolution quadrature are respectively used to approximate the Laplace operator and Riemann-Liouville fractional derivative. We first derive the priori estimate of the solution, and then present the error estimates of the space semi-discrete and the fully discrete schemes. The extensive numerical experiments validate the effectiveness of the numerical schemes.

\section*{Acknowledgements}
This work was supported by the National Natural Science Foundation of China under grant no. 11671182, and the Fundamental Research Funds for the Central Universities under grant no. lzujbky-2018-ot03.

\end{document}